\documentclass[10pt]{amsart}
\RequirePackage[colorlinks,citecolor=blue,urlcolor=blue,linkcolor=blue]{hyperref}
\hypersetup{
colorlinks = true,
citecolor=blue,
urlcolor=blue,
linkcolor=blue,
pdfpagemode = UseNone
}
\usepackage{graphicx,xspace,colortbl,url,hyperref}
\usepackage{amsmath,amsthm,amsfonts,setspace,tikz}
\usepackage{color}
\usepackage{fancybox}
\usepackage{epsfig}
\usepackage{subfig}
\usepackage{pdfsync}
\usepackage{enumerate}
\usepackage{geometry,soul}
    \def\qed{\hfill$\sqcap\kern-8.0pt\hbox{$\sqcup$}$\\}
    \def\beq{\begin{eqnarray}}
    \def\eeq{\end{eqnarray}}
    \def\beqq{\begin{eqnarray*}}
    \def\eeqq{\end{eqnarray*}}

      \def\PP{{\mathbb P}}
    
    \def\EE{{\mathbb E}}
    \def\r{{\mathbb R}}

    \def\ind{{\mathbb I}}

    \def\vv{{\textnormal v}}

        \def\calS{{\mathcal S}}
         \def\calI{{\mathcal I}}
          
\newtheorem{theorem}{Theorem}[section]
\newtheorem{lemma}[theorem]{Lemma}
\newtheorem{proposition}[theorem]{Proposition}

\theoremstyle{definition}
\newtheorem{definition}{Definition}[section]

\newtheorem{remark}[theorem]{Remark}

\makeatletter
\renewcommand*\env@matrix[1][\arraystretch]{%
  \edef\arraystretch{#1}%
  \hskip -\arraycolsep
  \let\@ifnextchar\new@ifnextchar
  \array{*\c@MaxMatrixCols c}}
\makeatother

\usepackage{stmaryrd}

\numberwithin{equation}{section}
\newcommand{\la}{\lambda}
\newcommand{\eps}{\varepsilon}

  \newcommand{\comment}[1]{\ovalbox{\footnotesize \color{magenta}#1}}

  \newcommand{\hide}[1]{ \comment{\tiny ..hidden text ..}}
\renewcommand{\hide}[1]{\comment{begin suppressed text} #1  \comment{end suppressed text}}
\def\vv#1{{\boldsymbol #1}}

\def\wt#1{\widetilde{#1}}

\newcommand{\RR}{\r}

\newcommand{\NN}{\mathbb{N}}

\def\Var{{\rm Var}}

\def\topp#1{^{(#1)}}

\newcommand{\C}{  \mathsf c }
\newcommand{\A}{  \mathsf a}
\newcommand{\B}{ \mathsf b}

\newcommand{\U}{   \mathsf u  }
\newcommand{\V}{    \mathsf v  }

\newcommand{\ff}{F}
\newcommand{\fa}{A}

\newcommand{\floor}[1]{\left\lfloor #1 \right\rfloor}
\title[Two-line   representation  of stationary measure for open TASEP ]
{A two-line   representation of stationary measure for open TASEP }
\author{W{\l}odek Bryc}
\address
{
W\l odzimierz Bryc\\
Department of Mathematical Sciences\\
University of Cincinnati\\
2815 Commons Way\\
Cincinnati, OH, 45221-0025, USA.
}
\email{wlodek.bryc@gmail.com}
\author{Pavel  Zatitskii}
\address
{
Pavel  Zatitskii\\
Department of Mathematical Sciences\\
University of Cincinnati\\
2815 Commons Way\\
Cincinnati, OH, 45221-0025, USA.
}
\email{zatitspl@ucmail.uc.edu}


 \usepackage{framed,pgf}
\newcounter{oldeq}
\newcounter{usesofarxiv}
 \newcommand{\arxiv}[1]{
\setcounter{oldeq}{\value{equation}}
 \addtocounter{usesofarxiv}{1}
 \setcounter{equation}{0}
\def\theoldeq{\theequation}
\def\theequation{x-\arabic{usesofarxiv}.\arabic{equation}}
\def\theequation{\arabic{section}.\arabic{usesofarxiv}.\arabic{equation}}
\def\theequation{\thesection.\arabic{usesofarxiv}.\arabic{equation}}
  \colorlet{shadecolor}{gray!10}
{\footnotesize
\begin{shaded}#1
\end{shaded}
   \setcounter{equation}{\value{oldeq}}
\numberwithin{equation}{section}
}\color{black}}
\renewcommand{\hide}[1]{}
\begin{document}
\maketitle

\begin{abstract}
We show that the stationary measure for the totally asymmetric simple exclusion
process on a segment with open boundaries is given by a marginal
of a two-line measure with a simple and explicit description.
We use this representation to analyze asymptotic
fluctuations of the height function near the triple point for a larger set of parameters than was previously studied. As a second application, we determine a single expression for the rate function in the
large deviation principle for the height function in the fan and in the shock region. We then discuss how this expression relates to the expressions for the rate function available in the literature.
\end{abstract}

\arxiv{This is an expanded version of the paper with additional details.}
\section{Introduction}
A totally asymmetric simple exclusion process (TASEP) with open boundaries is a continuous-time finite-state Markov process that models the movement of particles along the $N$ sites  $\{1,\dots,N\}$  from the left reservoir to the right reservoir. The particles cannot occupy the same site, and can move only to the nearest site on the right at rate $1$. The particles arrive at the first location, if empty, at rate $\alpha>0$ and leave the system from the $N$-th site, if occupied, at rate $\beta>0$.
For a description of the infinitesimal generator of this Markov process, we refer to e.g. \cite[Section 3]{Liggett-1975}. We will be interested solely in the stationary measure of this process.

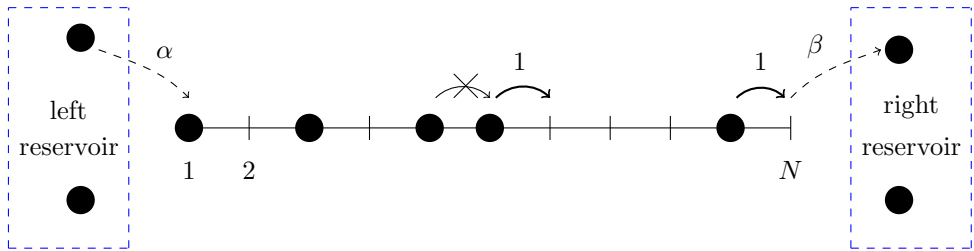
\begin{figure}[htb]
  \begin{tikzpicture}[scale=.8]
     \draw[-] (0.5,1) to (10.5,1);
\draw [fill=black, ultra thick] (.5,1) circle [radius=0.2];
  \draw[-] (1.5,.8) to (1.5,1.2);
\draw [fill=black, ultra thick] (2.5,1) circle [radius=0.2];
  \draw [fill=black,ultra thick] (4.5,1) circle [radius=0.2];
   \draw [fill=black, ultra thick] (5.5,1) circle [radius=0.2];

   \node [above] at (-1.5,1.) {left};
   \node [below] at (-1.5,1.) {reservoir};
      \draw[-,dashed,blue] (-2.5,-1) to (-2.5,3);
    \draw[-,dashed,blue] (-0.5,-1) to (-0.5,3);
        \draw[-,dashed,blue] (-2.5,3) to (-0.5,3);
             \draw[-,dashed,blue] (-2.5,-1) to (-0.5,-1);
      \node [above] at (12.5,1.) {right };
       \node [below] at (12.5,1.) {reservoir};
      \draw[-,dashed,blue] (11.5,-1) to (11.5,3);
       \draw[-,dashed,blue] (11.5,3) to (13.5,3);
       \draw[-,dashed,blue] (11.5,-1) to (13.5,-1);
         \draw[-,dashed,blue] (13.5,-1) to (13.5,3);
      \draw [fill=black, ultra thick] (9.5,1) circle [radius=0.2];
     \draw[-] (10.5,.8) to (10.5,1.2);
     \draw[->,dashed] (-1,2.3) to [out=-20,in=135] (.5,1.5);
   \node [above right] at (-.2,2) {$\alpha$};
   \draw [fill=black, ultra thick] (-1.3,2.5) circle [radius=0.2];
     \draw[->,dashed] (10.5,1.5) to [out=45,in=200] (12,2.3);
     \node [above left] at (11.2,2) {$\beta$};
           \draw [fill=black, ultra thick] (12.3,2.3) circle [radius=0.2];

       \draw[-] (6.5,.8) to (6.5,1.2);
          \draw[-] (7.5,.8) to (7.5,1.2);
        \draw[-] (8.5,.8) to (8.5,1.2);
        \draw[-] (3.5,.8) to (3.5,1.2);

      \node [above] at (6.0,1.8) {$1$};
      \draw[->,thick] (5.6,1.5) to [out=45,in=135] (6.5,1.5);
        \draw[->] (4.6,1.5) to [out=45,in=135] (5.5,1.5);
         \draw[-] (4.9,1.5) to   (5.3,1.9);
          \draw[-] (4.9,1.9) to   (5.3,1.5);
                 \node [above] at (10,1.8) {$1$};
              \draw[->,thick] (9.6,1.5) to [out=45,in=135] (10.4,1.5);
     \draw [fill=black, ultra thick] (-1.3,-.2) circle [radius=0.2];
    \node [above] at (0.5,0) {$1$};
    \node [above] at (1.5,0) {$2$};
          \node [above] at (10.5,0) {$N$};
      \draw [fill=black, ultra thick] (12.3,-.2) circle [radius=0.2];
\end{tikzpicture}
  \caption{Totaly Asymmetric Simple Exclusion process with boundary parameters $\alpha,\beta$.}\label{Fig1}
\end{figure}
The stationary measure for open TASEP has been studied for a long time, with explicit expressions available in \cite{schutz1993phase}
and \cite{derrida93exact}.
In this paper we establish a two-line representation for this stationary measure in terms of a pair of weighted random walks.
 We remark that there are numerous other representations for the stationary measure of  TASEP; a  representation   in \cite[Section 5.2]{nestoridi2023approximating} does not separate the "two lines", but it covers a more general ASEP. { Integral representation of the probability generating function  \cite{Bryc-Wesolowski-2015-asep,wang2023askey} is useful in studying Laplace transformations of limiting fluctuations.}
  Ref. \cite[Section 3.2]{duchi2005combinatorial} represents the stationary measure of a sequential TASEP as a marginal of a
   "two-layer" configuration that has a different form than ours.

The Gibbs measure (or line ensemble) representations have been valuable in studying integrable
probabilistic models on full or half-space and have been extended to time-homogeneous models on an interval with two-sided boundary conditions in \cite{barraquand2024stationary}.
Barraquand, Corwin, and Yang \cite[Theorem 1.3]{barraquand2024stationary} establish that stationary measures for the free-energy increment process of geometric last passage percolation on a diagonal strip are described as marginals of explicitly defined two-layer Gibbs measures. They further pose the question of obtaining an explicit description of the open TASEP stationary measure from its implicit connection to the stationary solution of the exponential large passage percolation recurrence. This paper proposes an alternative approach implicitly based on the matrix method \cite{Derrida-DEHP-1993}. We demonstrate that a representation akin to their two-layer Gibbs measure holds for the stationary measure of the TASEP.
We then use this representation to analyze asymptotic fluctuations of particle density for a larger set of parameters than was previously studied, and to prove the large deviations principle with a single expression for the rate function valid for all $\alpha,\beta\in(0,1)$.

We now introduce configuration spaces and probability measures that will facilitate canonical representations of the random variables that we need. We begin with the stationary measure of TASEP which defines a (discrete) probability measure $\PP_{\mathrm{TASEP}}$ on the configuration space $\Omega=\Omega\topp{ N}=\{0,1\}^N$. (We will omit the superscript $N$ when it is fixed in an argument and clearly recognizable from the context.) We assign probability
 $\PP_{\mathrm{TASEP}}(\vv \tau)$ to a sequence $\vv\tau=(\tau_1,\dots,\tau_N)\in \Omega$ that encodes the occupied and empty sites,
  where $\tau_j\in\{0,1\}$  is the occupation indicator of the $j$-th site.
It will be convenient to parameterize $\PP_{\mathrm{TASEP}}$ using
\begin{equation}\label{AC=..}
 \A=\frac{1-\alpha}{\alpha},\quad \B =\frac{1-\beta}{\beta}.
\end{equation}
Formula \eqref{AC=..}  makes sense for all $\alpha,\beta>0$,
and then $\PP_{TASEP}$ is determined by $\A ,\B>-1$ in formula \eqref{P:TASEP}, but in this paper we will only consider $\alpha,\beta\in(0,1)$, so that $\A ,\B> 0$.

  The steady state {\em height function} $\vv H_N$ is defined by
\begin{equation}\label{HN}
  H_N(k)=\tau_1+\dots +\tau_k, \quad k=0,1,\dots,N.
\end{equation}
   The invariant law
 $\PP_{\mathrm{TASEP}}$ is uniquely determined by the law $\PP_{\mathrm{H}}\topp{N}$  induced by $\vv H_N$ on the configuration  space
\begin{equation*}
  \label{calS}
\calS:=\left\{\vv s=(s_j)_{0\leq j\leq N}:\; s_0=0,\; s_j-s_{j-1}\in\{0,1\},\; 1\leq j\leq N \right\}.
\end{equation*}
Indeed,
$$\PP_{\mathrm{H}}\topp{N}(\vv s)=\PP_{\mathrm{TASEP}}(\vv \tau)
$$ with unique $\vv \tau$ such that $\vv s=\vv H_N(\vv \tau)$, as $\vv H_N:\Omega\to\calS$ is a bijection.
Instead of determining $\PP_{\mathrm{TASEP}}$, we will therefore determine $\PP_{\mathrm{H}}\topp{N}$ as a marginal law of the top line of the two-line   ensemble on the configuration space $\calS\times\calS$.

Denote by $\PP_{\mathrm{rw}}$,
the uniform law on $\calS\times\calS$  defined by two independent random walks with i.~i.~d. $\mathrm{Bernoulli}$  increments $\Pr(\xi=0)=\Pr(\xi=1)=1/2$,
$$\PP_{\mathrm{rw}}(\vv s_1,\vv s_2)=\Pr(\vv S_1=\vv s_1)\Pr(\vv S_2=\vv s_2)=1/4^N, \quad \vv s_1,\vv s_2\in\calS.$$
The {\em two-line ensemble} (TLE)  is  the probability measure $\PP_{\mathrm{TLE}}$  on $\calS\times\calS $,   defined
as follows:
\begin{equation}\label{P-Gibbs}
  \PP_{\mathrm{TLE}}\left(\vv s_1,\vv s_2\right)=\frac{1}{\mathfrak{C}_{\A ,\B }} \; \frac{\B  ^{s_1(N)-s_2(N)}}{(\A \B )^{\min_{0\leq j\leq N} \left\{s_1(j)-s_2(j)\right\}}} \PP_{\mathrm{rw}}(\vv s_1,\vv s_2),
\end{equation}
where $\mathfrak{C}_{\A ,\B }$ is the normalization constant and $\vv s_1,\vv s_2\in \calS\topp N$.
We will write $\EE_{\mathrm{TLE}}$ for the expected value with respect  to $\PP_{\mathrm{TLE}}$.
Of course, $\PP_{\mathrm{TLE}}=\PP_{\mathrm{TLE}}\topp N$ depends on $N$.

The two canonical coordinate mappings $\vv S_1,\vv S_2:\calS\times\calS \to \calS$ given by
\begin{equation}\label{S-coords}
  \vv S_1(\vv s_1,\vv s_2)=\vv s_1, \quad  \vv S_2(\vv s_1,\vv s_2)=\vv s_2,
\end{equation}
give a canonical realization on $(\calS\times\calS ,\PP_{\mathrm{rw}})$ of the pair of independent Bernoulli random walks $\vv S_1,\vv S_2$ and at the same time they give
a realization of the  two-line ensemble on $(\calS\times\calS ,\PP_{\mathrm{TLE}})$.

Our main result is the following TASEP analog of \cite[Theorem 1.3]{barraquand2024stationary}.
\begin{theorem}\label{thm1.1}
  For $\A ,\B>0 $, the marginal law of   random sequence $\vv S_1$ given by \eqref{S-coords}
  under measure  $\PP_{\mathrm{TLE}}$ is the (unique) law of the steady state height function
  $\vv H_N$ of the TASEP with parameters $\alpha,\beta$ given by \eqref{AC=..}.
   That is, for $\vv s \in \calS$,
  \begin{equation}\label{GE2H}
    \PP_{\mathrm{H}}(\vv s)=\sum_{\vv s'\in  \calS}    \PP_{\mathrm{TLE}}(\vv s,\vv s').
  \end{equation}

\end{theorem}
\begin{remark}
  As pointed out to us by an anonymous reviewer, Theorem \ref{thm1.1} in fact  holds for $\A,\B\geq 0$, i.e., for  $\alpha,\beta\in(0,1]$. The only change that is required is to rewrite  \eqref{P-Gibbs} as an expression in $\A,\B$ with non-negative exponents:
  $$\frac{1}{\mathfrak{C}_{\A ,\B }} \;  \B  ^{s_1(N)-s_2(N)-\min_{0\leq j\leq N} \left\{s_1(j)-s_2(j)\right\}}\A ^{\max_{0\leq j\leq N} \left\{s_2(j)-s_1(j)\right\}} \PP_{\mathrm{rw}}(\vv s_1,\vv s_2).$$
  Then the identities  that we  establish in the proof for $\A,\B>0$ extend to $\A,\B\geq 0$ by continuity.
\end{remark}

The proof of Theorem \ref{thm1.1} appears in Section \ref{Sec:proofT11}. In Section \ref{Sec:Appl} we give two applications which show how
Theorem \ref{thm1.1} allows to deduce asymptotic of the height function of TASEP  from well known
 asymptotic properties of random walks.
In Theorem \ref{Thm2.1} we use Theorem \ref{thm1.1} and Donsker's theorem to obtain convergence of the fluctuations of TASEP to the process conjectured to be a stationary measure of a KPZ fixed point on an interval in the full range of parameters. To our knowledge, previously available results of this form, see \cite[Theorem 1.5]{Bryc-Wang-Wesolowski-2022}, required that the sum $\U+\V$ of the parameters
in \eqref{uv2ab} be non-negative. (On the other hand, the more general five-parameter ASEP was covered.)
In Theorem \ref{Cor:LDP4H}  we show that in  the case of TASEP  the large deviation principle
for the height function is a consequence of Theorem \ref{thm1.1},  Mogulskii's theorem, and the contraction principle.
The large deviation principle  for the height function of a more general ASEP has been
analyzed in \cite{derrida2003-Exact-LDP}, but besides the simplicity of the proof,
a slight novelty here is the unified proof and an expression for the rate function, which works for all
$\alpha,\beta\in(0,1)$, i.e., for all $\A,\B>0$, in the so called fan region  $\A\B<1$ and in the shock region  $\A\B>1$.
Since the relation of our rate function to
 formulas \cite[(1.7), (3.3), and (1.11)]{derrida2003-Exact-LDP}
  is not obvious, we discuss this topic in Section \ref{Sect:old-LDP}.

\section{Proof of Theorem \ref{thm1.1}}\label{Sec:proofT11}
For $\A\B<1$, Theorem \ref{thm1.1} can be deduced from \cite[Section 2.3]{Barraquand2023Motzkin} used with $q=0$. The general case can then be obtained by analytic  continuation as discussed in \cite[Remark 1.9]{barraquand2024stationary}.
 Our more direct proof is based on induction on the size $N=1,2,\dots$ of the system and relies on a recursion for the invariant probabilities.
 Recursions for the invariant probabilities of a more general open asymmetric simple exclusion process (ASEP) appear in
\cite{Liggett-1975},   \cite{Derrida-Domany-Mujamel-1992}, and \cite{Derrida-Evans-1993}. Here, we will use a recursion that arises directly from the celebrated matrix method developed in \cite{derrida93exact}.
This recursion appears under the name {\em basic weight equations} in \cite[Theorem 1]{brak2006combinatorial} and it has already been used for similar purposes in \cite[Section 2.2]{nestoridi2023approximating}.
Specified to TASEP, the basic weight equations say that the unique stationary measure of TASEP under reparameterization \eqref{AC=..}
is given by
\begin{equation}
  \label{P:TASEP}
  \PP_{\mathrm{TASEP}}(\vv \tau)=\frac{1}{\mathsf{Z}_{\A,\B}} p_N(\vv \tau), \quad \vv \tau\in \Omega\topp{ N},
\end{equation}
where $\{ p_N(\vv \tau)\}$ satisfy
 the recursion that determines the un-normalized steady state weights $p_N$ uniquely
 in terms of the steady state weights $p_{N-1}$ for a TASEP on $\Omega\topp{N-1}$.
With the initial conditions
 \begin{equation}\label{p1}
    p_1(0)=  1+\A ,\quad p_1(1)= 1+\B,
  \end{equation} for $N\geq 2$ the recursion is:
  \begin{equation}\label{p0tau}
    p_N(0,\tau_2,\dots,\tau_N)=  (1+\A)p_{N-1}(\tau_2,\dots,\tau_N),
  \end{equation}
   \begin{equation}\label{ptau1}
    p_N(\tau_1,\dots,\tau_{N-1},1)=  (1+\B ) p_{N-1}(\tau_1,\dots,\tau_{N-1}).
  \end{equation}
  \begin{multline}
    \label{p10}
     p_N(\tau_1,\dots,\tau_{n-1},1,0,\tau_{n+2},\dots,\tau_N) =
     p_{N-1}(\tau_1,\dots,\tau_{n-1},1,\tau_{n+2},\dots,\tau_N)
     \\ +p_{N-1}(\tau_1,\dots,\tau_{n-1},0,\tau_{n+2},\dots,\tau_N),
  \end{multline}
$1\leq n\leq N-1$.

\arxiv{For completeness we
 re-derive the recursion  directly from the matrix method.
 The matrix method appropriate for the TASEP with parameters $\alpha,\beta\in(0,1)$ is given in terms of a pair of bidiagonal matrices and a pair of vectors \cite[(36), (37)]{derrida93exact} which in our notation become
$$
\mathbf{D}=\begin{bmatrix}
1+\B  & \sqrt{1-\A \B} & 0 & 0 & \dots\\
0 & 1 & 1 & 0 &\dots &\\
0& 0&1 & 1 & \dots &\\
\vdots & & &\ddots & \ddots \\
\end{bmatrix}, \quad
\mathbf{E}=\begin{bmatrix}
1+\A & 0 & 0 & 0 & \dots\\
\sqrt{1-\A \B} & 1 &  0 &\dots &\\
 0&1 & 1 &  0 &\dots &\\
 0& 0&1 & 1 &   \dots \\
\vdots  & & &\ddots & \ddots \\
\end{bmatrix}
$$
$ W=[1,0, 0,\dots ]$, $V=W^T$.
Here $\sqrt{1-\A  \B}$ should be interpreted as the imaginary number $i \sqrt{\A  \B-1}$ when $\A \B>1$.
  A  calculation verifies that the following commutation and eigenvalue properties    hold:
\begin{equation}
  \label{DEHP}
  \mathbf{D}\mathbf{E}= \mathbf{D}+\mathbf{E},\quad
  \langle W |\mathbf E= (1+\A)\langle W |, \quad \mathbf{D}|V\rangle=(1+\B )|V\rangle.
\end{equation}
It is then known, see \cite{derrida93exact}, that
\begin{equation}
  \label{DEHP-1}
  p_N(\tau_1,\dots,\tau_N):=\langle W | \prod_{j=1}^N(\tau_j\mathbf{D}+(1-\tau_j)\mathbf{E})|V\rangle
\end{equation}
 gives the un-normalized stationary probabilities for the TASEP.
 Although some of the entries of the matrices $\mathbf{E},\mathbf{D}$ may be imaginary,
it is well known that the resulting probabilities are real and strictly positive when $\A ,\B>-1$. This precludes any difficulties with normalization.

 The expression \eqref{DEHP-1} yields the "basic weight equations" as follows.
 The last two relations in \eqref{DEHP} give the initial values \eqref{p1} that allow us to start the recursion.
For $N\geq 2$, $(\tau_1,\dots,\tau_N)\in \Omega\topp{N}$, the last two relations in \eqref{DEHP} give \eqref{p0tau} and \eqref{ptau1}.  For $1\leq n\leq N-1$  the first relation $\mathbf{D}\mathbf{E}= \mathbf{D}+\mathbf{E}$ gives \eqref{p10}
with the natural omission of the empty sequences that arise when $n=1$ or $n=N-1$.
}
\subsection{The key identity and the proof of Theorem \ref{thm1.1}}
We introduce a family of functions  $f_N: \Omega\topp {N}\to (0,\infty)$  that we will use to prove \eqref{GE2H}.
For $\vv \tau=(\tau_1,\dots,\tau_N)\in \Omega\topp {N}$ and $\vv \xi =(\xi_1,\dots,\xi_N)\in \Omega\topp{N}$, let \begin{equation}
  \label{s2tau} s_1(k)=\sum_{j=1}^k \tau_j,\quad s_2(k)=\sum_{j=1}^k \xi_j
\end{equation}
with $s_1(0):=0$ and $s_2(0):=0$. Formula \eqref{s2tau} defines a pair of bijections $\Omega\topp{ N}\to \calS$  and  throughout this proof we will treat
$\vv s_1=\vv s_1(\vv \tau)$ as a function of $\vv \tau$ and $\vv s_2=\vv s_2(\vv \xi)$ as a function of $\vv \xi$.

With the above convention, we introduce
\begin{equation}\label{fN}
 f_N(\vv \tau)= f_N(\tau_1,\dots,\tau_N)=\sum_{(\xi_1,\dots,\xi_N)\in \Omega\topp{N}} \frac{\B ^{s_1(N)-s_2(N)}}{(\A \B)^{\min_{0\leq j \leq N}\{s_1(j)-s_2(j)\}}}.
\end{equation}

Theorem \ref{thm1.1} is a consequence of the following identity:
\begin{lemma}
  \label{Lem:H=mard}
   For $(\tau_1,\dots,\tau_N)\in \Omega\topp{N}$ and $N\geq 1$, we have
  \begin{equation}
    \label{p=f}
   f_N(\tau_1,\dots,\tau_N)=  p_N(\tau_1,\dots,\tau_N),
  \end{equation}
  with $p_N(\vv \tau)$ from representation \eqref{P:TASEP}.
\end{lemma}
 \begin{proof} It is clear that \eqref{p=f} holds for $N=1$. Indeed, in this case \eqref{fN} is the sum of two terms corresponding to $\xi_1=0$ and $\xi_1=1$:
   $$f_1(0)= \frac{\B^0}{(\A \B)^0} + \frac{\B^{-1}}{(\A \B)^{-1}}= 1+\A,$$
    $$f_1(1)= \frac{\B^1}{(\A \B)^0} + \frac{\B^0}{(\A \B)^{0}} =\B+1, $$
    matching the initial conditions \eqref{p1}.

Next, we show that $f_N$ satisfies the same three recursions as $p_N$  for $N\geq 2$.
Throughout this proof, for $k=0,\dots,N$ we consider partial sums $\vv s_1\in\calS\topp N$
and $\wt {\vv s}_1 \in \calS\topp{N-1}$ that depend on $\tau_1,\dots,\tau_n$ and partial sums $\vv s_2 \in\calS\topp N$ and $\wt{\vv  s_2}\in\calS\topp{N-1}$ that depend on the auxiliary $\{0,1\}$-valued  variables $\xi_1,\dots,\xi_n$ that appear under the sum in \eqref{fN}. In one place in the last part of the proof,
the sequence $\wt{\vv s}_1$ will be an explicit function of both sequences $\vv \tau$ and $\vv \xi$.
We express $f_N$ in terms of $\vv s_1, \vv s_2\in\calS\topp N$ and relate it to $f_{N-1}$ written in terms of $\wt{\vv s}_1, \wt{\vv s}_2\in\calS\topp{N-1}$.

First, we verify that
\begin{equation*}\label{f:0tau}
   f_N(0,\tau_2,\dots,\tau_N)=(1+\A) f_{N-1}(\tau_2,\dots,\tau_N).
\end{equation*}
To see this, we define $\wt s_1(k)=\sum_{j=1}^k \tau_{j+1}$ and $\wt s_2(k)=\sum_{j=1}^k \xi_{j+1}$ so that $s_1(N)=\wt s_1(N-1)$, $s_2(N)=\xi_1+\wt s_2(N-1)$ and
$\min_{0\leq j\leq N} \left\{s_1(j)-s_2(j)\right\}=  -\xi_1+\min_{0\leq j\leq N-1} \left\{\wt s_1(j)-\wt s_2(j)\right\}$. Then \eqref{fN} gives
\begin{multline*}
  f_N(0,\tau_2,\dots,\tau_N)=
  \sum_{(\xi_1,\xi_2,\dots,\xi_N) \in \{0,1\}^{N} }
   \frac{\B^{-\xi_1}}{(\A\B)^{-\xi_1}} \cdot \frac{\B ^{\wt s_1(N-1)-\wt s_2(N-1)}}{(\A \B)^{\min_{0\leq j \leq N-1}\{\wt s_1(j)-\wt s_2(j)\}}}
   \\
  = \sum_{\xi_1\in\{0,1\}} \A^{\xi_1}
  \sum_{(\xi_2,\dots,\xi_N) \in \{0,1\}^{N-1} }
    \frac{\B ^{\wt s_1(N)-\wt s_2(N)}}{(\A \B)^{\min_{0\leq j \leq N-1}\{\wt s_1(j)-\wt s_2(j)\}}}
    \\=(1+\A) f_{N-1}(\tau_2,\dots,\tau_N).
\end{multline*}
Next, by a similar argument we verify that
\begin{equation*}\label{f:tau1}
  f_N(\tau_1,\dots,\tau_{N-1},1)=(1+\B) f_{N-1}(\tau_1,\dots,\tau_{N-1}).
\end{equation*}
In this case, we introduce $\wt s_1(k)=\sum_{j=1}^k \tau_j$ and $\wt s_2(k)=\sum_{j=1}^k \xi_j$ so that
$s_1(N)=\wt s_1(N-1)+1$, $s_2(N)= \wt s_2(N-1)+\xi_N$. We  note that
$$s_1(N)-s_2(N) =\wt s_1(N-1)-\wt s_2(N-1)+1-\xi_N\geq \wt s_1(N-1)-\wt s_2(N-1),$$ so in this case
$\min_{1\leq j\leq N}\{s_1(j)-s_2(j)\}= \min_{1\leq j\leq N-1}\{\wt s_1(j)-\wt s_2(j)\}$. Thus
\eqref{fN} gives
\begin{multline*}
  f_N(\tau_1,\dots,\tau_{N-1},1)=\sum_{\xi_N\in\{0,1\}}
  \sum_{\xi_1,\dots,\xi_{N-1}\in \{0,1\}^{N-1}} \B^{1-\xi_N}\frac{ \B^{\wt s_1(N-1)-\wt s_2(N-1)}}{(\A\B)^{\min _{1\leq j\leq N-1} \{\wt s_1(j)-\wt s_2(j)\}}}
  \\
  =\sum_{\xi_N\in\{0,1\}}
 \B^{1-\xi_N}  \sum_{\xi_1,\dots,\xi_{N-1}\in \{0,1\}^{N-1}} \frac{\B^{\wt s_1(N-1)-\wt s_2(N-1)}}{(\A\B)^{\min _{1\leq j\leq N-1} \{\wt s_1(j)-\wt s_2(j)\}}}
 \\ = (1+\B) f_{N-1}(\tau_1,\dots,\tau_{N-1}).
\end{multline*}
Finally, we verify that for a fixed $1\leq n \leq N-1$ we have
 \begin{multline}
    \label{f10}
     f_N(\tau_1,\dots,\tau_{n-1},1,0,\tau_{n+2},\dots,\tau_N) =
     f_{N-1}(\tau_1,\dots,\tau_{n-1},1,\tau_{n+2},\dots,\tau_N)
     \\ +f_{N-1}(\tau_1,\dots,\tau_{n-1},0,\tau_{n+2},\dots,\tau_N).
  \end{multline}
  Here for $k\leq n-1$ we  let $\wt s_1(k)=s_1(k)$ and $\wt s_2(k)=s_2(k)$. For $n\leq k \leq N-1$ we set
  $\wt s_2(k)=s_2(n-1)+\sum_{j=n}^{k}\xi_{j+1}$, skipping over $\xi_n$. On the other hand, for $n\leq k \leq N-1$
  we let  $\wt s_1(k)=s_1(n-1)+(1-\xi_n)+\sum_{j=n+1}^{k}\tau_{j+1}$.
  (This is one place in the proof where $\wt {\vv s}_1$ depends on both $\vv \tau$ and $\vv \xi$.)
   Note that putting this choice of $\wt {\vv s}_1$ into expression \eqref{fN}
   leads to the formula for  $f_{N-1}(\tau_1,\dots,\tau_{n-1}, 1-\xi_n, \tau_{n+2},\dots,\tau_N)$.

  It is clear that
  \begin{multline}\label{s1-s2}
    s_1(N)-s_2(N)=\left(\sum_{j=1}^{n-1}\tau_j +1+0+\sum_{j=n+2}^N \tau_j\right) - \left(\sum_{j=1}^{n-1} \xi_j +\xi_n +\sum_{j=n+1}^N \xi_j\right)
    \\
    =\left(\sum_{j=1}^{n-1}\tau_j +(1-\xi_n)+\sum_{j=n+1}^{N-1} \tau_{j+1}\right) -\left( \sum_{j=1}^{n-1} \xi_j   +\sum_{j=n}^{N-1} \xi_{j+1}\right)
    =\wt s_1(N-1)-\wt s_2(N-1).
  \end{multline}

  The same calculation shows that
  $s_1(k)-s_2(k) =\wt s_1(k-1)-\wt s_2(k-1) $ for $k=n+1,\dots,N$.
  Since $s_1(k)-s_2(k) =\wt s_1(k)-\wt s_2(k) $ for $0\leq k\leq n-1$,
  and by the same rewrite as in \eqref{s1-s2} we get
  $$s_1(n)-s_2(n)=1-\xi_n+s_1(n-1)-s_2(n-1)\geq s_1(n-1)-s_2(n-1), $$
   we see that $s_1(n)-s_2(n)$ does not contribute to the minimum. This shows that the two minima are the same,
  \begin{multline*}
     \min_{0\leq k \leq N}\{s_1(k)-s_2(k)\} =
     \min_{0\leq k \leq n-1}\{s_1(k)- s_2(k)\} \wedge \min_{n+1\leq k \leq N}\{ s_1(k)- s_2(k)\}
     \\ =\min_{0\leq k \leq n-1}\{\wt s_1(k)- \wt s_2(k)\} \wedge \min_{n+1\leq k \leq N}\{ \wt s_1(k-1)- \wt s_2(k-1)\} = \min_{0\leq k \leq N-1}\{\wt s_1(k)- \wt s_2(k)\}.
  \end{multline*}

Therefore summing in \eqref{fN} over $\xi_1,\dots,\xi_{n-1},\xi_{n+1},\dots,\xi_{N}\in\{0,1\}$ and isolating the sum over $\xi_n\in \{0,1\}$ we get
$$
  f_N(\tau_1,\dots,\tau_{n-1},1,0,\tau_{n+2},\dots,\tau_N) =\sum_{\xi_n\in\{0,1\}}
     f_{N-1}(\tau_1,\dots,\tau_{n-1},1-\xi_n,\tau_{n+2},\dots,\tau_N),
$$
which establishes \eqref{f10}.

Since $f_N(\tau_1,\dots, \tau_N)$ and $p_N(\tau_1,\dots,\tau_N)$ satisfy the same recursion with respect to $N$ and the same initial conditions at $N=1$, this ends the proof.
 \end{proof}

\begin{proof}[Proof of Theorem \ref{thm1.1}]
For $\vv s_1,\vv s_2\in \calS$, denote %
\begin{equation}\label{g(s)}
  g_N(\vv s_1,\vv s_2)= \frac{\B ^{s_1(N)-s_2(N)}}{(\A \B)^{\min_{0\leq j \leq N}\{s_1(j)-s_2(j)\}}},
\end{equation}
and let
$\hat g_N(\vv \tau,\vv \xi)$ denote the same expression treated as a function of $\vv\tau,\vv\xi\in\Omega\topp{N}$ under the bijection \eqref{s2tau}. (That is, we apply \eqref{g(s)} to  $\vv s_1$ which is the height function of $\vv \tau$ and to $\vv s_2$, which is the height function of $\vv \xi$.)
In this notation,
\eqref{fN} becomes
\begin{equation}
  \label{f=sum-h}
  f_N(\vv \tau)=\sum_{\vv \xi\in\Omega\topp{N}} \hat g_N(\vv \tau,\vv \xi).
\end{equation}

Since $\PP_{\mathrm{rw}}$ is a uniform  law on $\calS\times \calS$,   formula \eqref{P-Gibbs} can be written as
\begin{equation*}
  \label{P-Gibbs-2}
  \PP_{\mathrm{TLE}}(\vv s_1,\vv s_2)=\frac{1}{\mathsf Z_{\A,\B}}g_N(\vv s_1,\vv s_2),
\end{equation*}
where $\mathsf Z_{\A,\B}$ is the same normalizing constant as in \eqref{P:TASEP}. Indeed, the normalizing constant is
\begin{equation*}
  \sum_{\vv s_1,\vv s_2\in \calS} g_N(\vv s_1,\vv s_2) = \sum_{\vv \tau,\vv \xi \in \Omega_{T}}
 \hat g_N(\vv \tau,\vv \xi)=  \sum_{\vv \tau \in \Omega_{T}} f_N(\vv \tau)=
 \sum_{\vv \tau \in \Omega_{T}} p_N(\vv \tau)=\mathsf Z_{\A,\B},
\end{equation*}
where we used \eqref{f=sum-h} and Lemma \ref{Lem:H=mard}.

Thus, writing $\vv \tau=\vv \tau(\vv s_1)$ and $\vv \xi=\vv \xi(\vv s_2)$ for the inverse of bijection \eqref{s2tau}, we have
\begin{multline*}
\sum_{\vv s_2\in \calS}  \PP_{\mathrm{TLE}}(\vv s_1,\vv s_2)=
\frac{1}{\mathsf Z_{\A,\B}}\sum_{\vv s_2\in \calS}g_N(\vv s_1,\vv s_2)
=\frac{1}{\mathsf Z_{\A,\B}}\sum_{\vv \xi\in \Omega}\hat g_N(\vv \tau(\vv s_1),\vv \xi)
\\
=\frac{1}{\mathsf Z_{\A,\B}}f_N(\vv \tau(\vv s_1))
=\frac{1}{\mathsf Z_{\A,\B}}p_N(\vv \tau(\vv s_1)) = \PP_{\mathrm{TASEP}}(\vv \tau(\vv s_1))
=\PP_{\mathrm{H}}(\vv s_1),
\end{multline*}
proving \eqref{GE2H}.
\end{proof}
\section{Applications}\label{Sec:Appl}
 In this section, we use Theorem \ref{thm1.1} to refine and extend existing results concerning the asymptotic behavior of the height function in steady state. Our two theorems draw inspiration and build upon the foundational works of \cite{derrida04asymmetric} where non-Gaussian fluctuations were identified, and \cite{derrida2003-Exact-LDP} which described the rate function for large deviations.

\subsection{Stationary measure of the conjectural KPZ fixed point on a segment}

Barraquand and Le Doussal \cite{barraquand2022steady}  introduced process
$B+X$, where $B,X$ are independent processes, $B$ is the Brownian motion of  variance $1/2$, and the law of $X$ is given by the Radon-Nikodym derivative
\begin{equation}\label{RN-B}
  \frac{d\PP_X}{d\PP_B}(\omega)=\frac{1}{\mathfrak K(\U,\V)} \exp \left((\U+\V)\min_{0\leq x\leq 1}\omega(x)-\V \omega(1)\right), \quad \omega\in C[0,1].
\end{equation}
They proposed this process as the stationary measure of the conjectural KPZ fixed point on the interval $[0,1]$ with boundary parameters $\U,\V\in\RR$, and predicted that the process $B+X$ should arise as a scaling limit of stationary measures of all models in the KPZ universality
class on an interval.

 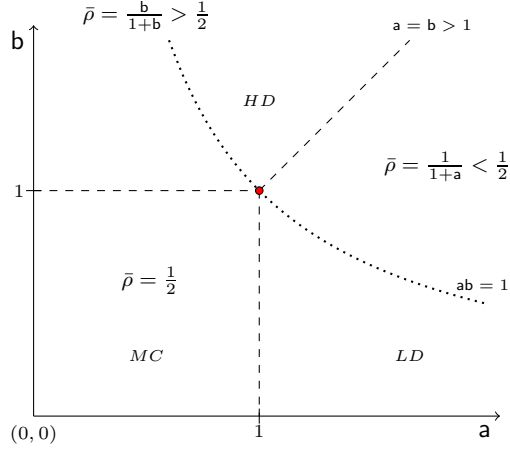
\begin{figure}[hbt]
  \begin{tikzpicture}[scale=1.0]
 \draw[->] (5,5) to (5,10.2);
 \draw[->] (5.,5) to (11.2,5);
   \draw[-, dashed] (5,8) to (8,8);
   \draw[-, dashed] (8,8) to (8,5);
   \draw[-, dashed] (8,8) to (10,10);
   \node [left] at (5,8) {\scriptsize$1$};
   \node[below] at (8,5) {\scriptsize $1$};
     \node [below] at (11,5) {$\A$};
   \node [left] at (5,10) {$\B$};

\draw[scale = 1,domain=6.8:11,smooth,variable=\x,dotted,thick] plot ({\x},{1/((\x-7)*1/3+2/3)*3+5});

\node[right] at (10.5,6.75) {\tiny $\A\B=1$};

\node[above] at (10.3,10) {\tiny $\A=\B>1$}; %

  \draw[-] (8,4.9) to (8,5.1);
   \draw[-] (4.9,8) to (5.1,8);

 \node [below] at (5,5) {\scriptsize$(0,0)$};
    \node[above] at (8,9) {\tiny $HD$};
    \node [below] at (10,6) {\tiny $LD$}; %

 \node [below] at (6.5,6) {\tiny $MC$};%
 \node [above] at (6.5,6.5) {\footnotesize { $\bar\rho=\tfrac12$}};%

\node [above] at (6.5,10) {\footnotesize  {$\bar\rho=\tfrac{\B}{1+\B}>\frac12$}};%

\node [above] at (10.5,8) {\footnotesize {$\bar\rho=\tfrac{1}{1+\A}<\frac12$}};%

\draw [fill=red] (8,8) circle [radius=0.05];
\end{tikzpicture}
\caption{
For fixed $\A,\B$, limiting particle density
$\bar \rho:=\displaystyle \lim_{N\to\infty}\tfrac1N(\tau_1+\dots+\tau_N)$ varies by region of the
phase        diagram for the open TASEP.
   These are the maximal current region, marked as $MC$,   the low density region, marked as $LD$,
and the high density region, marked as $HD$. Hyperbola $\A\B=1$ separates the {\em fan region} from the {\em shock region} $\A\B>1$.
 (These regions were identified in \cite[Fig. 3]{Derrida-Domany-Mujamel-1992}.) Parameters $\A_N,\B_N$ of the TASEP in Theorem \ref{Thm2.1} vary with system size $N$ and converge   to the  {\em  triple point} $(\A,\B)=(1,1)$, where  the  regions $MC$, $LD$, and $HD$  meet.
\label{Fig2}}
\end{figure}

This prediction is
supported by the limit theorem established in \cite{Bryc-Wang-Wesolowski-2022}, which demonstrated that the process $B+X$ describes the asymptotic behavior of the fluctuations of the height function of ASEP with parameters that vary with the size of the system according to formula \eqref{uv2ab} under condition $\U+\V>0$.
In this paper, we extend the above limit theorem to all real values of $\U$ and $\V$, removing the positivity constraint $\U+\V>0$ for TASEP. This, to our knowledge, constitutes the first confirmation of the prediction in \cite{barraquand2022steady} that encompasses the entire range of real parameters $\U$ and $\V$. A recent preprint  \cite{wang2024asymmetric} uses our result in combination with analytic techniques to extend Theorem \ref{Thm2.1} to the general ASEP.

  Recall that the height function is defined by \eqref{HN}.
\begin{theorem} \label{Thm2.1}
Consider a sequence of TASEPs indexed by $N$ on the segments $\{1,\dots,N\}$ for all $N\in\NN$, with parameters $\alpha=\alpha(N)\to 1/2$, $\beta=\beta(N)\to 1/2$ as $N\to \infty$ at the rates given by relation \eqref{AC=..} with
\begin{equation}\label{uv2ab}
  \A=\A_N=e^{-\U/\sqrt{N}},\quad \B=\B_N=e^{-\V/\sqrt{N}}.
\end{equation}

   Then
   \begin{equation*}
     \left\{\frac{1}{\sqrt{N}}\left( 2 H_N(\floor{x N})-\floor{N x}\right)\right\}_{x\in[0,1]}\Rightarrow \{B_x+X_x\}_{x\in[0,1]} \mbox{ as $N\to\infty$},
   \end{equation*}
   where the convergence is in Skorokhod's space $D[0,1]$ of c\`adl\`ag functions, processes $B,X$ are independent processes on $[0,1]$  with continuous trajectories, $B$ is a Brownian motion of variance $1/2$, and the law of $X$ is given by \eqref{RN-B}.
\end{theorem}
We remark that a linear interpolation of the height function similar to the one that appears in  Theorem \ref{Cor:LDP4H} leads to the same conclusion under weak convergence in the space $C[0,1]$ with the supremum norm.
\begin{proof} Consider the coordinate processes
$\vv S_1,\vv S_2$ defined  on the probability space $(\calS\topp N\times\calS\topp N ,\PP_{\mathrm{TLE}})$ by \eqref{S-coords}.
We shall determine the limit of the process
\begin{equation}
  \label{S1-scaled}
 \vv W_1\topp N:= \frac{1}{\sqrt{N}}\left\{ 2 S_1(\floor{x N})-\floor{x N}\right\}_{x\in[0,1]}.
\end{equation}
We will establish a more general claim that under measure $\PP_{\mathrm{TLE}}$, we have joint convergence of the pair of processes:
\begin{multline}\label{2d-conv}
 \frac{1}{\sqrt{N}}\left( \{S_1(\floor{x N})+S_2(\floor{x N})-\floor{x N}\}_{x\in[0,1]}, \{S_1(\floor{x N})-S_2(\floor{x N})\}_{x\in[0,1]}\right)
 \\ \Rightarrow \left(\{B_x\}_{x\in[0,1]},\{X_x\}_{x\in[0,1]}\right),
\end{multline}
where $B,X$ are independent processes from the conclusion of the theorem.  Noting  that the process
\eqref{S1-scaled} is the sum of the processes on the left hand side of
\eqref{2d-conv}, this will end the proof.

   Fix a bounded continuous function %
   $\Phi:D([0,1];\RR^2)\to \RR$
   and write
   $$\vv W_\pm \topp N:=\tfrac12(\vv W_1\topp {N}\pm \vv W_2\topp N)$$ for the two processes on the left hand side of
\eqref{2d-conv}, where $\vv W_2\topp N$ is defined as in \eqref{S1-scaled} with $\vv S_2$ in place of $\vv S_1$.
Our goal is to prove that
\begin{equation}
  \label{goal0}
  \lim_{N\to\infty} \EE_{\mathrm {TLE}} \left[\Phi(\vv W_+\topp N,\vv W_-\topp N)\right]= \EE [\Phi(B,X)].
\end{equation}

Following the approach in \cite[ Proposition 1.3]{Bryc-Wang-2023b}, we use Donsker's theorem.
Since $\Var_{\mathrm {rw}}(S_1(N))=N/4$, by Donsker's theorem, under probability measure $\PP_{\mathrm {rw}}$, we have
$$
 (\vv W_1\topp N, \vv W_2\topp N) \Rightarrow  (W_1,W_2),   \quad \mbox{as $N\to\infty$},
$$
where $W_1,W_2$ are independent Wiener processes and convergence is in $D([0,1];\RR^2)$. Thus
$$(\vv W_+\topp N,\vv W_-\topp N)\Rightarrow (B,B'):=(W_1+W_2,W_1-W_2)/2,$$
where $B,B'$ are independent Brownian motions of variance $1/2$.
Using the Radon-Nikodym density \eqref{P-Gibbs} and \eqref{uv2ab}, we get
\begin{multline*}%
  \EE_{\mathrm {TLE}} \left[\Phi(\vv W_+\topp N,\vv W_-\topp N)\right]
  = \frac{1}{\mathsf{Z}_{\U,\V}(N)}
  \EE_{\mathrm {rw}} \Big[\Phi(\vv W_+\topp N,\vv W_-\topp N)
  \\
  \times \exp\left({\frac{\U+\V}{\sqrt{N}}\min_{0\leq j\leq N}(S_1(j)-S_2(j))-\frac{\V}{\sqrt{N}}(S_1(N)-S_2(N))}\right)
  \Big]
  \\=\frac{1}{\mathsf{Z}_{\U,\V}(N)}
  \EE_{\mathrm {rw}} \left[\Phi(\vv W_+\topp N,\vv W_-\topp N) e^{
  (\U+\V)\min_{0\leq x\leq 1}\vv W_-\topp N(x)-\V \vv W_-\topp N(1)}
  \right]
  \\=
  \frac{1}{\mathsf{Z}_{\U,\V}(N)}
  \EE_{\mathrm {rw}} \left[\Phi(\vv W_+\topp N,\vv W_-\topp N) \mathcal{E}(\vv W_-\topp N) \right],
\end{multline*}
where
$$
\mathcal{E}(f)=e^{(\U+\V)\min_{0\leq x \leq 1}f(x)-\V f(1)}.
$$
Noting that $\Phi$ is bounded and by \eqref{Levy-Ott} and Remark \ref{Rem:howtouse}
$$\sup_N \EE_{\mathrm{rw}}\left[\mathcal{E}^2(\vv W_-\topp N)\right]<\infty,$$
 we see that the sequence of real valued random variables
  $$\left\{\Phi(\vv W_+\topp N,\vv W_-\topp N) \mathcal{E}(\vv W_-\topp N)\right\}_{ N=1,2,\dots}$$
is uniformly integrable with respect to $\PP_{\mathrm{rw}}$.
 Uniform integrability and weak convergence imply convergence of expectations (\cite[Theorem 3.5]{billingsley99convergence}),
 so Donsker's theorem implies that
 $$
 \lim_{N\to\infty} \EE_{\mathrm {rw}} \left[\Phi(\vv W_+\topp N,\vv W_-\topp N)
  \mathcal{E}(\vv W_-\topp N) \right]=\EE_{B,B'} \left[\Phi(B,B') \mathcal{E}(B') \right],
 $$
 where $\EE_{B,B'}$ denotes integration with respect to the law of $(B,B')$ on $C[0,1]\times C[0,1]$.
 Using this with $\Phi\equiv 1$, we see that the normalizing constants also converge,
 $$\lim_{N\to\infty}\mathsf{Z}_{\U,\V}(N)=\lim_{N\to\infty}\EE_{\mathrm {rw}} \left[  \mathcal{E}(\vv W_-\topp N) \right]
 = \EE_{B,B'} \left[ \mathcal{E}(B') \right]=\mathfrak{K}(\U,\V).$$
 Thus
 \begin{multline*}
    \lim_{N\to\infty} \EE_{\mathrm {TLE}} \left[\Phi(\vv W_+\topp N,\vv W_-\topp N)\right]
    = \frac{\lim_{N\to\infty} \EE_{\mathrm {rw}} \left[\Phi(\vv W_+\topp N,\vv W_-\topp N) \mathcal{E}(\vv W_-\topp N) \right]}{\lim_{N\to\infty} \mathsf{Z}_{\U,\V}(N)}
    \\
    =\frac{1}{\mathfrak{K}(\U,\V)} \EE_{B,B'}\left[\Phi(B,B') \mathcal{E}(B')\right]
    =\int_{C[0,1]} \left( \int_{C[0,1]} \Phi(b,b') \frac{\mathcal{E}(b')}{\mathfrak{K}(\U,\V)}  P_B(db')\right)  P_B(db)
    \\=
    \int_{C[0,1]}\left(\int_{C[0,1]}  \Phi(b,x)     P_X(dx) \right) P_B(db)
    =
    \EE_{B,X} \left[\Phi(B,X) \right],
 \end{multline*}
 where we used independence of the Brownian motions $B,B'$ and  \eqref{RN-B}.
(Here, $\EE_{B,X}$ denotes integration with respect to the (product) law of $(B,X)$ on $C[0,1]\times C[0,1]$.) This establishes \eqref{goal0} and ends the proof.

\end{proof}

\subsection{Large deviations}
Large deviations for the height function process $\{H_N(\floor{Nx})\}_{x\in[0,1]}$ of ASEP
have been discussed in Refs. %
 \cite{derrida2002exact} %
  and  \cite{derrida2003-Exact-LDP},
with the height  function interpreted as particle density profile. (There are also nice expositions in
\cite[Section 5]{derrida06matrix}, \cite[Section 16]{derrida07nonequilibrium}.)
In this section, we use Theorem \ref{thm1.1} to deduce large deviations for the height function of the TASEP directly from Mogulskii's theorem \cite[Theorem 5.1.2]{DZ-1998-large}. In Section \ref{Sect:old-LDP}
  we show how to recover formulas discovered in \cite{derrida2003-Exact-LDP}, and we determine the additive normalization.

Let $(\mathbb{X},d)$ be a complete separable metric space. Consider a sequence of probability spaces
$(\Omega\topp N,\PP\topp N)$ and a family of random variables $X_N:\Omega\topp N\to \mathbb{X}$, $N=1,2,\dots$. A standard statement of the large deviation principle in Varadhan's sense involves a family of Borel
 subsets $A$ of $\mathbb{X}$, their interiors $\rm{int}(A)$ and closures $\rm{cl}( A)$ and specifies
 asymptotics of probabilities in terms of a  rate function $I$ by the following upper/lower bounds:
\begin{multline*}
  -\inf_{x\in \rm{int}(A)} I(x)\leq \liminf_{N\to\infty}\frac1N \log \PP\topp N (X_N\in  A)
\\\leq \limsup_{N\to\infty}\frac1N \log \PP\topp N (X_N\in  A)\leq -\inf_{x\in\rm{cl}( A)} I(x).
\end{multline*}

It will be more convenient to use an equivalent definition which we now recall; see \cite[Theorem 4.4.13]{DZ-1998-large}.  (Compare also  \cite[Definitions 1.1.1 and 1.2.2]{Dupuis-Ellis1997weak}.)

\begin{definition}\label{D-LDP}
  Let $(\mathbb{X},d)$ be a complete separable metric space.   The sequence $\{X_N\}$   satisfies the {\em large deviation principle} (LDP), if
there exists a lower semicontinuous function $I:\mathbb{X}\to[0,\infty]$, called the {\em rate function}, such that \begin{enumerate}[(i)]
  \item for every bounded continuous function $\Phi:\mathbb{X}\to\RR$,
  \begin{equation}
    \label{Lap-Def}
    \lim_{N\to\infty}\frac1N\log \EE_N \left[\exp(N\Phi(X_N))\right] = \sup_{x\in \mathbb{X}}\{ \Phi(x)-I(x)\},
  \end{equation}
  where $\EE_N[\cdot]=\int_{\Omega\topp N} (\cdot) d\PP_N$ denotes the expected value.
  \item $I$ has compact level sets, $I^{-1}[0,w]$ is a compact subset of $\mathbb{X}$ for every $w\geq 0$.
\end{enumerate}
\end{definition}

To prove the LDP for the height function, we consider the sequence $\{X_N\}$ of $\mathbb{X}=C([0,1],\RR^2)$-valued random variables defined on
probability spaces $\Omega\topp N=(\calS\topp N\times \calS\topp N, \PP_{\mathrm{TLE}}\topp N)$ obtained by
linear interpolation between the points
$\left(\frac{k}{N},\frac{S_1(k)}{N}\right)$, $k=0,\dots,N$, in the first component and the points $\left(\frac{k}{N},\frac{S_2(k)}{N}\right)$, $k=0,\dots,N$, in the second component.
By Theorem \ref{thm1.1}, the first component of $X_N$ then has the same law as the continuous interpolation of
the height function \eqref{HN} based on the points $\left(\frac{k}{N},\frac{H_N(k)}{N}\right)$,  $k=0,\dots,N$. 

To introduce the rate function, we need additional notation. Let
$$h(x)=\begin{cases}
x \log x +(1-x)\log(1-x) & 0\leq x \leq 1, \\
\infty & x< 0 \mbox{ or } x > 1.
\end{cases}$$

Denote by $\mathcal{AC}_0$ the set of absolutely continuous functions $f\in C[0,1]$ such that $f(0)=0$.  Let
\begin{equation}
  \label{Lab}
  K(\A,\B) =   \log(\bar \rho(1-\bar\rho)),
\end{equation}
where
\begin{equation}
  \label{rho-bar}
\bar \rho=\bar\rho(\A,\B)=\begin{cases}
  \frac{1}{1+\A} & \A>1, \A<\B ,\\
  \frac12 & \A\le 1, \B\le 1, \\
  \frac{\B}{1+\B} & \B>1, \B>\A,
\end{cases}
\end{equation} denotes the limiting particle density, as indicated on the phase diagram in Fig. \ref{Fig2}.
(Notation $\log(\bar\rho(1-\bar\rho))$ was introduced
  in \cite{derrida2003-Exact-LDP}.)
\begin{theorem}\label{Thm:LDP} If  $\A,\B>0$, then
  the sequence  $\{X_N\}$ satisfies the large deviation principle with respect to the
  probability measures $\PP_{\mathrm{TLE}}\topp N$ with the rate function
  \begin{multline}
    \label{I}
    I(f_1,f_2)=
       \int_0^1 \left(h(f_1'(x)) +h(f_2'(x)) \right) dx
       \\
       + \log(\A\B)\ \min_{0\leq x\leq 1}(f_1(x)-f_2(x)) - (f_1(1)-f_2(1))\log \B-K(\A,\B)
  \end{multline}
  if functions $f_1,f_2\in\mathcal{AC}_0$; we let $I=\infty$ for all other $f_1,f_2\in C[0,1]$.
\end{theorem}
Since $h=\infty$ outside of $[0,1]$, it is clear that expression \eqref{I} can only be finite for
$f_1,f_2\in\mathcal{AC}_0$ with the derivatives in $[0,1]$ for almost all $x$.
In particular, as a consequence of the Arzel\`a-Ascoli theorem, $I(\cdot)$ is lower semicontinuous and has compact level sets.
 \begin{proof}
 By a theorem of Mogulskii, see \cite[Theorem 5.1.2]{DZ-1998-large}, $X_N$ satisfies the LDP with respect to the law  $\PP_{\mathrm{rw}}$ of two independent
 Bernoulli$(1/2)$  random walk paths.
The rate function for two independent components with Bernoulli$(1/2)$ increments becomes
 $$
 I_{\mathrm{rw}}(f_1,f_2)=\log 4+\int_0^1 \left(h(f_1'(x))+h(f_2'(x))\right)dx
 =\int_0^1 \left(h(f_1'(x)|\tfrac12)+h(f_2'(x)|\tfrac12)\right)dx,
 $$
 where  $f_1,f_2\in\mathcal{AC}_0$, see \cite[Exercise 2.2.23]{DZ-1998-large},
 and
\begin{equation}\label{H-entr}
  h(x|y)=
   x \log
   \left(\frac{x}{y}\right)+(1-x) \log
   \left(\frac{1-x}{1-y}\right), \quad x,y\in(0,1).
\end{equation}

In order to use \eqref{Lap-Def}, we need to  fix a bounded continuous function
$\Phi: C([0,1],\RR^2)\to\RR$  and compute
\begin{equation}\label{Lim2prove}
\lim_{N\to \infty}\frac1N \log \EE_{\mathrm{TLE}}\left[e^{N\Phi(X_N)} \right]=
 \lim_{N\to\infty} \frac1{N}\log \EE_{\mathrm{rw}}\left[e^{N\Phi(X_N)} g(X_N)^N
   \right],
\end{equation}
where  $g$  is given by \eqref{g(s)}.
To proceed, we
 introduce a version of $g$  that   acts on pairs of
continuous functions $f_j\in\mathcal{AC}_0$ by
$$\wt g (f_1,f_2):=\frac{\B^{f_1(1)-f_2(1)}}{(\A\B)^{\min_{0\leq x\leq 1}(f_1(x)-f_2(x))}}.
$$
It is clear that $\wt g(N X_N)=g(\vv S_1,\vv S_2)$, where $\vv S_k=(0, S_k(1),\dots,S_k(N))$ are the sums from \eqref{S-coords} %
and  that $\wt g(N X_N)=\wt g(X_N)^N=\exp(N \log \wt g(X_N))$.
 To compute the limit \eqref{Lim2prove}, we use Varadhan's lemma  with respect to $\PP_\mathrm{rw}$.
  According to Varadhan's lemma \cite[Theorem 4.3.1]{DZ-1998-large}, if $\Psi: C([0,1],\RR^2)\to\RR$ is a continuous function such that
  \begin{equation}
    \label{Exp-bd}
    \limsup_{N\to\infty}\frac1{N}\log \EE_{\mathrm{rw}}\left[e^{\gamma N\Psi(X_N)}
   \right]<\infty \mbox{ for some $\gamma>1$},
  \end{equation}
  then  $$
  \lim_{N\to\infty} \frac1{N}\log \EE_{\mathrm{rw}}\left[e^{N\Psi(X_N)}
   \right]=\sup_{f_1,f_2}\{\Psi(f_1,f_2)- I_{\mathrm{rw}}(f_1,f_2)\}.$$
   We apply this to a  continuous but unbounded function $\Psi: C([0,1],\RR^2)\to\RR$ given by
 \begin{multline*}
    \Psi(f_1,f_2)=\Phi(f_1,f_2)+\log \wt g (f_1,f_2)
    \\=
 \Phi(f_1,f_2)+\log (\B) (f_1(1)-f_2(1))-\log(\A\B) \min_{0\leq x\leq 1}(f_1(x)-f_2(x)).
 \end{multline*}
 To verify that \eqref{Exp-bd} holds with $\gamma=2$, we note that by the Cauchy-Schwartz inequality
 \begin{multline*}
 \EE_{\mathrm{rw}}\left[e^{2 N\Psi(X_N)}\right]
 =\EE_{\mathrm{rw}}\left[e^{2 N\Phi(X_N)} g^2(\vv S_1,\vv S_2)\right]
  \leq e^{N|\Phi\|_\infty} \EE_{\mathrm{rw}}\left[g^2(\vv S_1,\vv S_2)\right]
 \\ \leq   e^{N|\Phi\|_\infty}
 \left(\EE_{\mathrm{rw}}\left[(\A\B)^{-4 \min_{0\leq j\leq N}\{S_1(j)-S_2(j)\}  }\right]\right)^{1/2}
 \left(\EE_{\mathrm{rw}}\left[\B^{4(S_1(N)-S_2(N)) }\right]\right)^{1/2}.
 \end{multline*}
 Inequality \eqref{Levy-Ott}, see Remark \ref{Rem:howtouse}, shows that \eqref{Exp-bd} holds.
From Varadhan's lemma \cite[Theorem 4.3.1]{DZ-1998-large} we get
\begin{multline*}
   \lim_{N\to\infty} \frac1{N}\log \EE_{\mathrm{rw}}\left[e^{N\Phi(X_N)} \wt g(X_N)^N
   \right]
   =
   \sup_{f_1,f_2 \in C[0,1]}\{ \Phi(f_1,f_2)+\log \wt g(f_1,f_2)-I_{\mathrm{rw}}(f_1,f_2)\}
   \\
   =\sup_{f_1,f_2 \in C[0,1]} \Big\{
    \Phi(f_1,f_2)
     -\big(I_{\mathrm{rw}}(f_1,f_2)\\+\log (\A\B){\min_{0\leq x\leq 1} (f_1(x)-f_2(x))}- (f_1(1)-f_2(1))\log\B \big)\Big\}.
\end{multline*}
Using this with $\Phi\equiv 0$ we see that the normalizing constants in \eqref{P-Gibbs} satisfy $$ \lim_{N\to\infty} \frac1N\log \mathfrak{C}_{\A ,\B }(N)
=-K_0(\A,\B),$$
 where $K_0(\A,\B)$ is given by the same variational expression
     \begin{multline}\label{I_0}
    K_0(\A,\B)
    =\inf_{f_1,f_2\in\mathcal{AC}_0} \Big\{ \int_0^1 \left( h(f_1'(x)|\tfrac12) +h(f_2'(x)|\tfrac12) \right)dx
       \\
       +  \log(\A\B)\min_{0\leq x\leq 1}(f_1(x)-f_2(x)) - (f_1(1)-f_2(1))\log \B\Big\}.
  \end{multline}
    Thus
    \begin{multline*}%
      \lim_{N\to\infty} \frac1{N}\log \EE_{\mathrm{TLE}}\left[e^{N\Phi(X_N)}\right]
   =   \lim_{N\to\infty} \frac1{N}\log \EE_{\mathrm{rw}}\left[e^{N\Phi(X_N)} \wt g(X_N)^N
   \right]-\lim_{N\to\infty} \frac1N \log \mathfrak{C}_{\A ,\B }(N)
   \\
    =\sup_{f_1,f_2 \in C[0,1]} \Big\{
    \Phi(f_1,f_2) \hfill
    \\ -\left(I_{\mathrm{rw}}(f_1,f_2)+\log (\A\B){\min_{0\leq x\leq 1} (f_1(x)-f_2(x))}- (f_1(1)-f_2(1))\log\B - K_0(\A,\B)\right) \Big\}.
\end{multline*}
This proves LDP with the rate function that depends on $K_0(\A,\B)$.
To conclude the proof,
    it remains to verify  that   expression \eqref{I_0} for $K_0(\A,\B)$ simplifies to the expression $ K(\A,\B) =   \log(\bar \rho(1-\bar\rho))$ given in \eqref{Lab}.
  We postpone this part of the proof
 until
   Section \ref{Sect:old-LDP}, where  as part of the proof of Proposition \ref{Prop-LD-derrida-I},
we will obtain formula \eqref{Lab}  for $\A\B\geq 1$,
   and then  as part of the proof of Proposition \ref{Prop-B-W} we will establish that \eqref{Lab} holds also for $\A\B\leq 1$.
\end{proof}

{
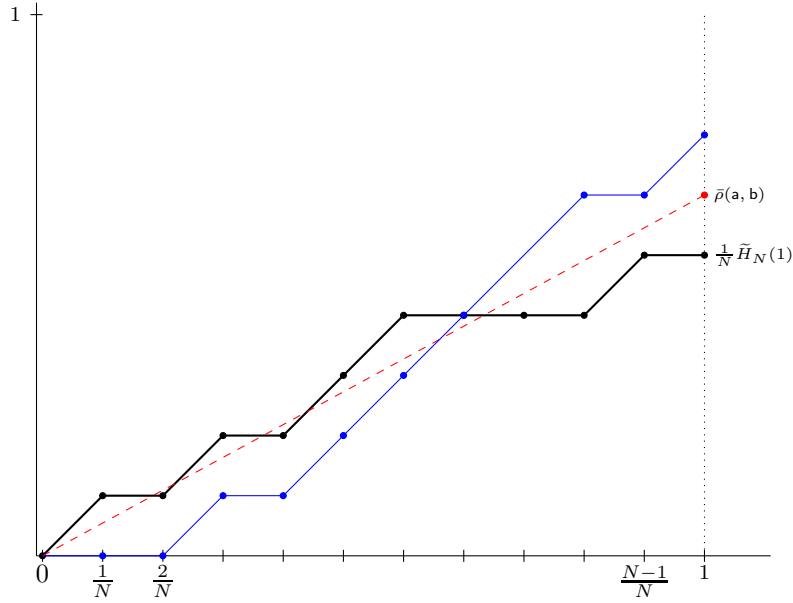
\begin{figure}
    \begin{tikzpicture}[scale=.8]

\draw[-] (.9,0) to (0.9,9.2);
   \draw[-] (.8,9) to (1,9);
      \node[left] at (.8,9) { \footnotesize $1$};
 \draw[-] (.9,0) to (13.1,0);
   \draw[-] (1,-.1) to (1,0.1);
  \draw[-] (2,-.1) to (2,0.1);
  \draw[-] (3,-.1) to (3,0.1);
  \draw[-] (4,-.1) to (4,0.1);
  \draw[-] (5,-.1) to (5,0.1);
  \draw[-] (6,-.1) to (6,0.1);
  \draw[-] (7,-.1) to (7,0.1);
  \draw[-] (8,-.1) to (8,0.1);
  \draw[-] (9,-.1) to (9,0.1);
  \draw[-] (10,-.1) to (10,0.1);
  \draw[-] (11,-.1) to (11,0.1);
  \draw[-] (12,-.1) to (12,0.1);

\draw [fill] (1,0) circle [radius=0.05];

\draw[-,thick] (1,0) to (2,1);
\draw [fill] (2,1) circle [radius=0.05];

\draw[-,thick] (2,1) to (3,1);
\draw [fill] (3,1) circle [radius=0.05];

\draw[-,thick] (3,1) to (4,2);
\draw [fill] (4,2) circle [radius=0.05];

 \draw[-,thick] (4,2) to (5,2);
\draw [fill] (5,2) circle [radius=0.05];

 \draw[-,thick] (5,2) to (6,3);
\draw [fill] (6,3) circle [radius=0.05];

 \draw[-,thick] (6,3) to (7,4);
\draw [fill] (7,4) circle [radius=0.05];

 \draw[-,thick] (7,4) to (8,4);
\draw [fill] (8,4) circle [radius=0.05];

 \draw[-,thick] (8,4) to (9,4);
\draw [fill] (9,4) circle [radius=0.05];
 \draw[-,thick] (9,4) to (10,4);
\draw [fill] (10,4) circle [radius=0.05];
 \draw[-,thick] (10,4) to (11,5);
\draw [fill] (11,5) circle [radius=0.05];
 \draw[-,thick] (11,5) to (12,5);
\draw [fill] (12,5) circle [radius=0.05];
   \node[below] at (1,0) {  $0$};
      \node[below] at (2,0) {  $\tfrac1N$};
       \node[below] at (3,0) {  $\tfrac2{N}$};

            \node[below] at (11,0) {  $\tfrac{N-1}{N}$};
             \node[below] at (12,0) { \footnotesize $1$};
\draw[-,dotted] (12,0) to (12,9);
\draw [fill,red] (12,6) circle [radius=0.05];
\node[right] at (12,6) { \tiny $\bar \rho(\A,\B)$};
\draw[-,dashed,thin,red] (1,0) to (12,6);

\node[right] at  (12,5) { \tiny $\frac1N\wt H_N(1)$};
\draw[-,blue] (1,0) to (3,0);
\draw[-,blue] (3,0) to (4,1);
\draw[-,blue] (4,1) to (5,1);
\draw[-,blue] (5,1) to (10,6);
\draw[-,blue] (10,6) to (11,6);
\draw[-,blue] (11,6) to (12,7);
\draw [fill,blue] (2,0) circle [radius=0.05];
\draw [fill,blue] (3,0) circle [radius=0.05];
\draw [fill,blue] (4,1) circle [radius=0.05];
\draw [fill,blue] (5,1) circle [radius=0.05];
\draw [fill,blue] (6,2) circle [radius=0.05];
\draw [fill,blue] (7,3) circle [radius=0.05];
\draw [fill,blue] (8,4) circle [radius=0.05];
\draw [fill,blue] (10,6) circle [radius=0.05];
\draw [fill,blue] (11,6) circle [radius=0.05];
\draw [fill,blue] (12,7) circle [radius=0.05];
\end{tikzpicture}
  \caption{ Continuous interpolation  $X_N$ is a pair of piecewise-linear lines. The first component $\frac1N\wt{H}_N$ of $X_N$ is marked as the thick black line and the second component of $X_N$ marked in blue.  Note that  $\frac1N\wt{H}_N(1)=\frac{1}{N}\sum_{j=1}^N\tau_j\to \bar\rho$, see Fig. \ref{Fig2}, except on the coexistence line with $\A=\B>1$. (The dashed line represents the most likely trajectory of the random curve $x\mapsto \frac1N\wt{H}_N(x)$ for large $N$.)
  \label{FigH}}
\end{figure}

}

It is clear that the continuous linear interpolation $\wt H_N$ of the step function forming the height
 function $\{H_N(\floor{Nx})\}_{x\in[0,1]}$ has the same law as the first  component of the vector $N X_N(x)$.
 By contraction principle (see e.g. \cite[Theorem 4.2.1]{DZ-1998-large} or \cite[Theorem  1.3.2]{Dupuis-Ellis1997weak}), this implies the following LDP.

\begin{theorem}\label{Cor:LDP4H}
    If $\A,\B>0$ then the sequence of linear interpolations $\{\frac1N \wt H_N\}$ of the height function of a TASEP on $\{1,\dots,N\}$ satisfies the large deviation principle with the rate function
   \begin{multline}\label{I(H)}
  \calI(f)= \inf_{g\in\mathcal{AC}_0}\Big\{  \int_0^1 (h(f'(x)) +h(g'(x))) dx
   \\ + \log(\A\B) \min_{0\leq x\leq 1}(f(x)-g(x)) -\log (\B)(f(1)-g(1)) \Big\} - \log (\bar \rho(1-\bar\rho))
\end{multline}
if $f\in\mathcal{AC}_0$, and $\calI(f)=\infty$ otherwise. Here $\bar \rho$ is given by \eqref{rho-bar}.
\end{theorem}
\begin{proof}
To obtain \eqref{I(H)}, we apply the contraction principle to the first coordinate mapping,
with \eqref{I} applied to $f_1=f$, $f_2=g$. Anticipating subsequent developments we wrote \eqref{Lab} for $K_0(\A,\B)$.
\end{proof}

A version of this result for a more general ASEP appears in \cite{derrida2003-Exact-LDP}, with the rate function rewritten in different and more explicit forms for $\A\B\leq 1$ and $\A\B>1$, which are discussed in Section \ref{Sect:old-LDP}.
The main novelty in Theorem \ref{Cor:LDP4H} is that its proof and the expression for the rate function do not distinguish between   the
shock  and the fan region. (This has been anticipated in \cite[Formula (3.56)]{Bertini-2007}.) On the other hand, additional nontrivial work is needed to recover the formulas that appear in \cite{derrida2003-Exact-LDP}.

We remark that Ref. \cite[Section 3.6]{Bertini-2007} uses a two layer representation \cite{duchi2005combinatorial} to obtain a version of   Theorem \ref{Cor:LDP4H} for the case $\A=\B=0$, which is not covered by our results.  Ref. \cite {Bertini-2007} also discusses the corresponding version of Proposition \ref{Prop-LD-derrida-II-a}.


\section{Comparison with previous large deviation results }\label{Sect:old-LDP}
In this section we discuss previous LDP results for ASEP, specialized to the case of
TASEP. With some additional work
these results
can be obtained from Theorem \ref{Cor:LDP4H}, and the derivations identify constant $K_0(\A,\B)$ given by the variational formula \eqref{I_0} in the proof as a simpler expression  \eqref{Lab} based on the phase diagram.

Converted to our notation, the LDP in Ref. \cite[(1.11)]{derrida2003-Exact-LDP} gives the following rate function for the shock region of TASEP:
\begin{proposition}
  \label{Prop-LD-derrida-I}
  If  $\A\B\geq 1$ then the rate function \eqref{I(H)} is
\begin{multline}
  \label{I-Derrida}
  \mathcal{I}(f)=\min_{0\leq  y \leq 1}\Big\{
  \int_0^y\left( f'(x)\log\tfrac{\A f'(x)}{1+\A}+(1-f'(x))\log \tfrac{1-f'(x)}{1+\A}\right)
   dx
   \\ +\int_y^1 \left(f'(x)\log \tfrac{f'(x)}{1+\B}+(1-f'(x))\log \tfrac{(1-f'(x))\B}{1+\B} \right)dx
   \Big\}-K(\A,\B),
\end{multline}
where
\begin{equation}
  \label{K1}
  K(\A,\B)=\log \min\left\{\frac{\A}{(1+\A)^2},\frac{\B}{(1+\B)^2}\right\} =\log \frac{\A\vee\B}{(1+\A\vee\B)^2}.
\end{equation}
\end{proposition}
Although this result is known, we provide a separate proof based on Theorem \ref{Cor:LDP4H}. This  allows us to  complete part of the postponed  proof of Theorem \ref{Cor:LDP4H},  where we need to show that $K_0(\A,\B)=K(\A,\B)= \log \bar \rho(1-\bar \rho)$ for $\A\B\geq1$.
To accomplish this goal,  we use \eqref{I(H)} with $K_0(\A,\B)$ given by \eqref{I_0}.
 (Then the fact that $ K_0(\A,\B) = \log \bar \rho(1-\bar \rho) $  for $\A\B\geq 1$ will follow
from \eqref{K1}.)
\begin{proof}[Proof of Proposition \ref{Prop-LD-derrida-I}]
With $\log(\A \B)\geq 0$,  we write \eqref{I(H)} as
\begin{multline*}
  \mathcal{I}(f)
   =- K_0(\A,\B)+\int_0^1 h(f'(x))dx
   \\+\inf_g \min_{0\leq y\leq 1}\Big\{  \int_0^1  h(g'(x)) dx
   + \log(\A\B) (f(y)-g(y)) -\log (\B)(f(1)-g(1)) \Big\}
   \\
   =- K_0(\A,\B)+ \int_0^1 h(f'(x)) dx +\min_{0\leq y\leq 1}\Big\{ f(y)\log(\A\B)-f(1)\log \B \\+
   \inf_{g_1\in\mathcal{AC}_0[0,y]} \big\{ \int_0^y h(g_1'(x)) dx   -\log(\A\B)  g_1(y)\big\}
   \\ +\inf_{g_2\in\mathcal{AC}[y,1]: g_2(y)=g_1(y)} \big\{ \int_y^1 h(g_2'(x)) dx    +\log (\B)g_2(1)\big\} \Big\},
\end{multline*}
where we split $g$ into $g_1=g|_{[0,y]}$, $g_2=g|_{[y,1]}$ for $0<y<1$, with (omitted) minor changes
 for $y=0$ or $y=1$.
Since  $\inf_{g}  \int_a^b h(g'(x)) dx$ is attained on linear functions, denoting
by $F=g_1(y)=g_2(y)$ and $G=g_2(1)$ the values of $g_2$ at the endpoints of interval $[y,1]$, the
optimal (linear) functions are
$\wt g_1(x)=F x/y$ and $\wt g_2(x)=\frac{(G-F) (x-y)}{1-y}+F$.
Optimizing over all possible choices of $F\leq G$ we get
  \begin{multline*}
     \inf_{g_1}  \int_0^y h(g_1'(x)) dx   -\log(\A\B)  g_1(y)+\inf_{g_2}  \int_y^1 h(g_2'(x)) dx    +\log (\B)g_2(1) \\=
     \min_{F\in[0,y]}\min_{G\in[F,F+1-y]}\{
     y h(F/y)-\log(\A\B)F+  (1-y) h(\tfrac{G-F}{1-y})+G\log \B\}
     \\=y \log \frac{1}{1+\A}+(1-y) \log \left(\frac{\B}{1+\B}\right),
   \end{multline*}
   as the minimum over $F,G$
    is attained at
   $$F= \frac{\A y}{1+\A},\quad G= %
   F+\frac{1-y}{1+\B}.$$
   We get
   \begin{multline*}
      \mathcal{I}(f)= - K_0(\A,\B)+\int_0^1 h(f'(x)) dx
      \\+\min_{y\in[0,1]}
      \Big\{ f(y) \log \A +y \log \tfrac{1}{1+\A}+(1-y) \log \tfrac{\B}{1+\B}-(f(1)-f(y))\log \B.
      \Big\}
      \\
      = - K_0(\A,\B)+\min_{y\in[0,1]}
      \Big\{ \int_0^y h(f'(x)) dx+f(y) \log \A +y \log \tfrac{1}{1+\A}
      \\+ \int_y^1 h(f'(x)) dx+(1-y) \log \tfrac{\B}{1+\B}-(f(1)-f(y))\log \B
      \Big\}.
   \end{multline*}
   To end the proof, we note that the two integrals under the minimum in \eqref{I-Derrida} match the two integrals in the expression above:
    \begin{multline*}
     \int_0^y\left( f'(x)\log\tfrac{\A f'(x)}{1+\A}+(1-f'(x))\log \tfrac{1-f'(x)}{1+\A}\right)
   dx=
   \int_0^y h(f'(x))dx \\+ f(y) \log \A
   + y \log \tfrac{1}{1+\A}
    \end{multline*}
    and
    \begin{multline*}
      \int_y^1 \left(f'(x)\log \tfrac{f'(x)}{1+\B}+(1-f'(x))\log \tfrac{(1-f'(x))\B}{1+\B} \right)dx
      = \int_y^1 h(f'(x))dx\\
      + (f(1)-f(y))\log \tfrac1{1+\B} +(1-y) \log \tfrac{\B}{1+\B}-(f(1)-f(y))\log \tfrac{\B}{1+\B}
      \\
       = \int_y^1 h(f'(x))dx+(1-y) \log \tfrac{\B}{1+\B}-(f(1)-f(y))\log \B.
    \end{multline*}
    The additive normalization constant $K_0(\A,\B)$  can now be determined
     from the condition that $\inf_f \mathcal{I}(f)=0$. To do so, we repeat the previous calculation again. Denoting by $F=f(y)$ and $G=f(1)$, we switch the order of the infima, and use the extremal property of linear functions again:
    \begin{multline*}
      K_0(\A,\B)= \min_{y\in[0,1]}
      \inf_f\Big\{\int_0^y h(f'(x)) dx  + \int_y^1 h(f'(x)) dx
      \\+f(y) \log \A +y \log \tfrac{1}{1+\A}+(1-y) \log \tfrac{\B}{1+\B}-(f(1)-f(y))\log \B.
      \Big\}
      \\=
       \min_{y\in[0,1]} \inf_{F\in[0,y],G\in[F,F+1-y]}
      \Big\{ y h(F/y)+ (1-y)h(\tfrac{G-F}{1-y})+ F \log \A
      \\+y \log \tfrac{1}{1+\A}+(1-y) \log \tfrac{\B}{1+\B}-(G-F)\log \B \Big\}.
     \end{multline*}
     The infimum is attained at $F=\frac{y}{1+\A}$, $G=F+\frac{\B}{1+\B}(1-y)$. Therefore,
     \begin{multline*} K_0(\A,\B)=
       \min_{y\in[0,1]} \Big\{ y \log \frac{\A}{(1+\A)^2} \log \left(\frac{\A}{\B}\right)-2 y \log (1+\A)+2 (y-1) \log (1+\B)+\log
   (\B)\Big\}
   \\
     = \min_{y\in[0,1]} \Big\{ y \log \left(\frac{\A}{(1+\A)^2}\right)-(1-y) \log \left(\frac{\B}{(1+\B)^2}\right)\Big\}
    \end{multline*}
 and \eqref{K1} follows.
This   establishes   \eqref{Lab}  for $\A\B\geq 1$.\end{proof}

For the fan region of TASEP, the LDP in \cite[(1.7), (3.3),  (3.6)]{derrida2003-Exact-LDP}
 gives the rate function which we recalculated as follows.

\begin{proposition}\label{Prop-LD-derrida-II-a}
 If $\A\B< 1$ then for $f\in \mathcal{AC}_0$ with $0\leq f'\leq 1$  the rate function \eqref{I(H)} is
\begin{equation}
  \label{I-Derrida2*}
  \mathcal{I}(f)=\int_0^1 h (f'(x))  dx + \int_0^1 (\wt f'(x) \log  G_*(x)+(1-\wt f')\log(1- G_*(x)))dx-K(\A,\B),
\end{equation}
where
$G_*(x)= \left(\wt{f}'(x) \vee \tfrac{\A}{1+\A}\right)\wedge \tfrac 1{1+\B},$
and $\wt f$ is  the convex envelope of $f$, i.e., the largest  convex function below $f$.
The normalizing constant is
\begin{equation}
  \label{K:ab<1}
  K(\A,\B)=\sup_{\tfrac{\B}{1+\B}\leq \rho\leq\tfrac{1}{1+\A}}\log \rho(1-\rho)= \begin{cases}
    \log \frac{\A}{(1+\A)^2} &  \A>1,  \\
    -2 \log 2&  \A,\B\leq 1, \\
   \log \frac{\B}{(1+\B)^2}& \B>1.
  \end{cases}
\end{equation}
\end{proposition}
We verify that this result follows from  Theorem  \ref{Cor:LDP4H}.
As previously, to avoid circular reasoning  we use \eqref{I(H)} with $K_0(\A,\B)$ given by \eqref{I_0}, without
identification $ K_0(\A,\B) = \log \bar \rho(1-\bar \rho) $. The identification will follow for $\A\B< 1$
once we  establish \eqref{K:ab<1} in the proof of Proposition \ref{Prop-B-W} below.

The proof is more substantial and requires additional lemmas, so we put it as a separate section.
\subsection{Proof of Proposition \ref{Prop-LD-derrida-II-a}}

In the proof, $f\in \mathcal{AC}_0$ with $0\leq f'\leq 1$ is fixed and $\wt f$ is its convex envelope. Since $\wt f$ is convex, $\wt f'$ is nondecreasing, and for definiteness, we take $\wt f'$  right-continuous.
 We write \eqref{I(H)} as $\mathcal{I}(f)=\int_0^1 h(f'(x))dx-K(\A,\B) +\inf_{g} J^*( f,g)$, where
\begin{equation}
  \label{Jg*}
J^*(f,g) = \int_0^1 h(g'(x)) dx + \log(\A\B) \min_{0\leq x\leq 1}(f(x)-g(x)) -\log (\B)(f(1)-g(1)).
\end{equation}
Similarly, we write \eqref{I-Derrida2*} as  $\mathcal{I}(f)=\int_0^1 h(f'(x))dx -K(\A,\B)+  J_*(\wt f, G_*)$, where
\begin{equation}
  \label{JG*}
 J_*(f,G) = \int_0^1\left[ f'(x) \log G(x) +(1-f'(x))\log(1-G(x))\right]dx.
\end{equation}
(One can verify that  $J_*(\wt f, G_*)= J_*(f, G_*)$, see the proof of  \cite[(A.5)]{derrida2003-Exact-LDP}, and it is the latter expression that
appears in the rate function \cite[(3.3)]{derrida2003-Exact-LDP}.)
\arxiv{The relation of \eqref{I-Derrida2*} with \cite[(1.7)]{derrida2003-Exact-LDP} is discussed in Appendix \ref{DLS2US}.}
We want to prove that
\begin{equation}\label{Pre-I}
    J_*(\wt f, G_*)=\inf_{g\in\mathcal{AC}_0} J^*(f,g).
\end{equation}

We first verify that
\begin{lemma}
\begin{equation}\label{Step4.5}
    \inf_{g\in\mathcal{AC}_0} J^*(f,g) = \inf_{g\in\mathcal{AC}_0} J^*(\wt f,g).
\end{equation}
\end{lemma}
\begin{proof}
    For any $g$ the inequality $J^*(f,g)\leq J^*(\tilde f,g)$ is trivial because $f\geq \wt f$, so
$\min(f-g) \geq \min(\wt f -g)$, $\log(\A\B)<0$, and $f(1)=\wt f(1)$.

For the converse inequality, consider the set $U = \{f> \wt f\}$. This set is open and therefore is a disjoint union of open intervals, say $U = \bigcup J_k$. On each of these intervals, $\wt f$ is linear.
Define a function $G$ as follows: $G(x) = g'(x)$ for $x\not \in U$ and
$$
G(x)=\frac{g(w_k)-g(u_k)}{w_k-u_k} \mbox{ for } x\in J_k=(u_k,w_k).
$$
Then $\int_{J_k} G= g(w_k)-g(u_k)$. Define
$\wt g(x)=\int_0^x G.$
We have
$$
\wt g(x)=\int_{[0,x]\setminus U} g'(x)+ \sum_{k} \ind_{\{w_k<x\}}(g(w_k)-g(u_k))+\ind_{\{x\in J_k\}} (x-u_k) \frac{g(w_k)-g(u_k)}{w_k-u_k},
$$
thus $\wt g$ is linear on each $J_k$ and $\wt g = g$ on $[0,1]\setminus U$. Then $\wt f - \wt g$ is also linear on $J_k$, and
\begin{multline*}
  \min_{J_k}(\wt f - \wt g)= \min(\wt f(u_k) - \wt g(u_k), \wt f(w_k) - \wt g(w_k))
  \\=
\min(f(u_k) - g(u_k), f(w_k) - g(w_k)) \geq \min_{[0,1]}(f-g).
\end{multline*}

Thus, $\min_{[0,1]}(\wt f -\wt g) \geq \min_{[0,1]}(f-g)$.
From convexity of $h$, we have
$$
\int_{J_k} h(g') \geq |J_k| h\Big(\frac{1}{|J_k|}\int_{J_k} g'\Big)=
|J_k| h\Big(\frac{g(w_k)-g(u_k)}{w_k-u_k}\Big)=
\int_{J_k} h(\wt g').
$$
Thus
$$
\int_0^1 h(g') \geq \int_0^1 h(\wt g').
$$
Therefore, we obtain
$
J^*(f,g) \geq J^*(\wt f,\wt g)$, completing the proof of \eqref{Step4.5}.
\end{proof}
Clearly, $\frac{\A}{1+\A}<\frac{1}{1+\B}$. For clarity,  we write $G_*$ in expanded form
\begin{equation}
    \label{*G*}
    G_*(x)=\begin{cases}
        \frac{\A}{1+\A} & \wt f '(x)<\frac{\A}{1+\A},  \\ \\
        \wt f '(x) &  \frac{\A}{1+\A} \leq \wt f '(x)<\frac{1}{1+\B}, \\ \\
        \frac{1}{1+\B} &   \wt f '(x)\geq \frac{1}{1+\B},
    \end{cases}
\end{equation}
and we note that $G_*$ is nondecreasing and right-continuous.

We can now relate the functionals $J_*$ and $J^*$.
\begin{lemma}\label{Step 3}  Let  $g_*(x)=\int_0^x G_*$. Then
    \begin{equation}
        \label{P3}
        J^*(\wt f,g_*)=J_*(\wt f,G_*).
    \end{equation}
\end{lemma}
\begin{proof}
With the convention $\inf \emptyset=1$ and $\sup \emptyset=0$, let
\begin{equation}
    \label{x1x2}
    x_1= \inf \left\{ x\geq 0:  \wt f '(x)\geq \frac{\A}{1+\A}\right\}, \quad x_2= \sup\left\{ x\leq 1:  \wt f '(x)<\frac{1}{1+\B} \right\}
\end{equation}
be the largest interval $[x_1,x_2)$ on which $G_*=\wt f'$. Under this convention,
we have $x_1=x_2=1$ when $\wt f'<\A/(1+\A)$  on $[0,1]$ and $x_1=x_2=0$ when  $\wt f'\geq 1/(1+\B)$ on [0,1]. It is also possible that $G_*$ jumps from its lowest to its largest value at some $x_*$ in which case we have $x_1=x_2=x_*$.

    From \eqref{x1x2} we see that $\wt f' < G_* = g_*'$ on $[0,x_1)$, $\wt f' = G_* = g_*'$ on $[x_1,x_2)$, and $\wt f' > G_* = g_*'$ on $(x_2,1]$. Then $\wt f - g_*$ decreases on $[0,x_1)$, is constant on $[x_1,x_2]$, and increases on $(x_2,1]$. Thus the minimal value of $\wt f - g_*$ is attained on the entire interval $[x_1,x_2]$.

Identity \eqref{P3} now follows by direct calculation.
We have
\begin{multline}\label{*J*}
    J_*(\wt f,G_*)=\int_0^{x_1} \left(\wt f' \log \frac{\A}{1+\A}+(1-\wt f')\log \frac{1}{1+\A}\right)
    +\int_{x_1}^{x_2} h(\wt f')
    \\+ \int_{x_2}^1 \left(\wt f' \log\frac{1}{1+\B}+(1-\wt f')\log \frac{\B}{1+\B}\right)\\
    =x_1 \log \frac1{1+\A}+\wt f(x_1) \log \A +  \int_{x_1}^{x_2} h(\wt f') +(1-x_2)\log \frac{\B}{1+\B} -(\wt f(1)-\wt f(x_2))\log \B .
\end{multline}
On the other hand, since the minimum of $\wt f - g_*$ is attained at all points of $[x_1,x_2]$, we get
\begin{multline}\label{*min}
   \log(\A\B) \min_{[0,1]}(\wt f-g_*)-(\wt f(1)-g_*(1))\log\B
\\ =(\wt f(x_1)-g_*(x_1))\log\A -(\wt f(1)-\wt f(x_2))\log\B+ (g_*(1)-g_*(x_2))\log\B.
\end{multline}
 Since
$g_*(x_1)=x_1 \frac{\A}{1+\A}$ and $g_*(1)-g_*(x_2)=(1-x_2)\frac1{1+\B}$,
 \begin{multline}\label{*int}
      \int_0^1 h(g_*')=\int_0^{x_1} h\left(\frac{\A}{1+\A}\right)+\int_{x_1}^{x_2} h(\wt f')+ \int_{x_2}^1 h\left(\frac{1}{1+\B}\right)
 \\= x_1 \log \frac{1}{1+\A}+x_1\frac{\A}{1+\A}\log\A +\int_{x_1}^{x_2} h(\wt f')- (1-x_2)\frac{1}{1+\B}\log\B +(1-x_2) \log \frac{\B}{1+\B}
 \\= x_1 \log \frac{1}{1+\A}+g_*(x_1)\log\A+\int_{x_1}^{x_2} h(\wt f')- (g_*(1)-g_*(x_2))\log\B+(1-x_2) \log \frac{\B}{1+\B}.
 \end{multline}
Combining \eqref{*min} and \eqref{*int}, we get
$$
J^*(\wt f, g_*)=x_1 \log \frac{1}{1+\A}+\wt f(x_1)\log\A +\int_{x_1}^{x_2} h(\wt f')+(1-x_2) \log \frac{\B}{1+\B}-(\wt f(1)-\wt f(x_2))\log\B
$$
which ends the proof by \eqref{*J*}.
\end{proof}

With \eqref{Step4.5} and \eqref{P3} at hand, to complete the proof of \eqref{Pre-I}, we need to show that for any $g\in \mathcal{AC}_0$ the following inequality holds:
\begin{equation}\label{*Pavel3}
 J^*(\wt f,g)\geq J^*(\wt f, g_*),
  \end{equation}
hence, by \eqref{Step4.5}, the infimum $\inf_{g} J^*(f,g)=\inf_{g} J^*(\wt f,g)$ is attained at $g_*$.

To prove \eqref{*Pavel3}, we fix $g\in\mathcal{AC}_0$ and consider the function
$$
\ff(\tau) = J^*(\wt f, \tau g + (1-\tau)g_*), \qquad  \tau \in [0,1].
$$
Note that this function is convex because the functional $J^*(\wt f, \,\cdot\,)$ is convex. We claim that  the right-derivative $\partial_+\ff(0) \geq 0$. When this is proved, the convexity implies that $\ff$ is increasing on $[0,1]$, and therefore
$$
 J^*(\wt f, g) =  \ff(1)\geq \ff(0) =  J^*(\wt f, g_*).
$$

We need the following technical lemma.
\begin{lemma}\label{lem1}
    Let $\phi$ and $\psi$ be  continuous functions on $[0,1]$. Let $\fa  = \{x \in [0,1]\colon \phi(x) = \min\limits_{[0,1]} \phi\}$. Then
    $$
    \min_{[0,1]}(\phi - \tau \psi) = \min_{[0,1]}\phi - \tau \cdot\max_{\fa}\psi +o(\tau) \quad  \mbox{ as } \tau \to 0^+.
    $$
\end{lemma}
\begin{proof}
First, the set $\fa$ is closed and therefore $\psi$ reaches the maximum on it, say, at $z$. Then for $\tau \geq 0$ we have
$$
\min_{[0,1]}(\phi-\tau\psi) \leq \phi(z)-\tau \psi(z) = \min_{[0,1]}\phi - \tau \cdot\max_{\fa}\psi.
$$
To prove the converse estimate, let us fix any $\eps_1>0$ and using the continuity of $\psi$ find $\delta>0$ such that
$
\max_{\fa_\delta}\psi \leq \max_{\fa}\psi + \eps_1,
$
where $\fa_\delta = \{y \in [0,1]\colon \mathrm{dist}(y,\fa)< \delta\}$.
Then use continuity of $\phi$ to find $\eps_2>0$ such that
$\min_{[0,1]\setminus \fa_\delta}\phi>\min_{\fa}\phi + \eps_2$.
For a small positive $\tau$ we have
\begin{align*}
    \min_{[0,1]\setminus \fa_\delta}(\phi-\tau\psi) \geq \min_{\fa}\phi + \eps_2 - \tau \max_{[0,1]}\psi \geq \min_{[0,1]}\phi - \tau \max_{\fa}\psi;\\
    \min_{\fa_\delta}\ (\phi-\tau\psi) \geq \min_{\fa}\phi - \tau \max_{\fa_\delta}\psi \geq \min_{[0,1]}\phi - \tau \max_{\fa}\psi - \tau\eps_1.
    \end{align*}
Combining these inequalities, we obtain
$$
\liminf_{\tau \to 0+} \frac{\min_{[0,1]}(\phi-\tau\psi) - \min_{[0,1]}\phi}{\tau}\geq -\max_{A}\psi - \eps_1.
$$
Tending $\eps_1$ to zero, we finish the proof.
\end{proof}

 \begin{proof}[Proof of Proposition \ref{Prop-LD-derrida-II-a}]
We apply Lemma~\ref{lem1} with $\phi = \wt f-g_*$ and $\psi = g-g_*$ to obtain
$$
\min_{[0,1]}(\wt f - (\tau g + (1-\tau)g_*)) = \min_{[0,1]}(\wt f - g_* - \tau (g -g_*)) =  \min_{[0,1]}(\wt f - g_*) - \tau \max_{\fa}(g -g_*) + o(\tau), \;\tau \to 0^+,
$$
where $\fa = \{x \in [0,1]\colon \wt f(x) - g_*(x) = \min\limits_{[0,1]} (\wt f - g_*)\}=[x_1,x_2]$.
We use this formula to calculate the right derivative of $\ff$ at $0$: \footnotesize
\begin{align}
\partial_+\ff(0) =& \int_0^1 (g'-g_*')\log\Big(\frac{g_*'}{1-g_*'}\Big)- \log(\A\B)\max_{[x_1,x_2]}(g -g_*) +(g(1)-g_*(1))\log \B \notag\\
=& \int_0^1 (g'-g_*')\log\Big(\frac{\B g_*'}{1-g_*'}\Big) - \log(\A\B)\max_{[x_1,x_2]}(g -g_*)\notag\\
=& \int_0^{x_1} (g'-g_*')\log\Big(\frac{\B \frac{\A}{\A+1}}{1-\frac{\A}{\A+1}}\Big)+
\int_{x_1}^{x_2} (g'-g_*')\log\Big(\frac{\B g_*'}{1-g_*'}\Big)+
\int_{x_2}^{1} (g'-g_*')\log\Big(\frac{\B \frac{1}{\B+1}}{1-\frac{1}{\B+1}}\Big)
\notag \\ &- \log(\A\B)\max_{[x_1,x_2]}(g -g_*)\notag\\
=& \log(\A\B)\int_0^{x_1} (g'-g_*')+
\int_{x_1}^{x_2} (g'-g_*')\log\Big(\frac{\B g_*'}{1-g_*'}\Big)- \log(\A\B)\max_{[x_1,x_2]}(g -g_*)\notag\\
=&\int_{x_1}^{x_2} (g'-g_*')\log\Big(\frac{\B g_*'}{1-g_*'}\Big)- \log(\A\B)\Big(\max_{[x_1,x_2]}(g -g_*) - \big(g(x_1)-g_*(x_1)\big)\Big).\label{eq4}
\end{align} \normalsize
If $x_1=x_2$, this gives  $\partial_+\ff(0) =0$. If $x_1<x_2$, we proceed as follows.

Consider the function $\Phi=\log\Big(\frac{\B g_*'}{1-g_*'}\Big)$   on $[x_1,x_2]$.
It is a nondecreasing right-continuous function
with values $\Phi(x_1)\geq \log(\A\B)$ and $\Phi(x_2)\leq 0$. Write
$\Phi(x)-\Phi(x_1)=\mu([x_1,x])$, where $\mu$ is a non-negative
measure of total mass $\mu([x_1,x_2])\in [0, - \log (\A\B)]$  on Borel subsets of $[x_1,x_2]$.

Write $\psi=g'-g'_*$ and $\Psi(x)=\int_{x_1}^x \psi = g(x)-g_*(x)-(g(x_1)-g_*(x_1))$.

Since  $\Phi(x)-\Phi(x_1)=\mu([x_1,x])$ for $x\in[x_1,x_2]$ and $\Psi(x_1)=0$, Fubini' theorem
gives
\begin{multline*}
   L:=\int_{x_1}^{x_2}\psi \Phi=\int_{x_1}^{x_2} \psi(x)\left(\Phi(x_1)+ \int_{[x_1,x_2]} 1_{t\leq x} d\mu(t)\right)dx
\\ =\Psi(x_2)\Phi(x_1) +\int_{[x_1,x_2]} \int_t^{x_2} \psi(x)dxd\mu(t)
\\=\Psi(x_2)\Phi(x_1) +\int_{[x_1,x_2]}(\Psi(x_2)-\Psi(t))d\mu(t)
=\Psi(x_2)\Phi(x_2)-\int_{[x_1,x_2]} \Psi(t)d\mu(t).
\end{multline*}

Since $\Phi(x_2)\leq 0$, we get $\Psi(x_2)\Phi(x_2)\geq \Phi(x_2)\max_{[x_1,x_2]}\Psi$, therefore
$$
L = \Psi(x_2)\Phi(x_2) - \int_{[x_1,x_2]} \Psi(z) \ d\mu(z) \geq \left(\Phi(x_2)-\mu([x_1,x_2])\right) \max_{[x_1,x_2]}\Psi .
$$
Thus
$$\int_{x_1}^{x_2}\psi \Phi \geq \Phi(x_1)\max_{[x_1,x_2]}\Psi.$$
Returning back to \eqref{eq4}, we see that since $\max_{[x_1,x_2]}\Psi\geq \Psi(x_1)=0$, and
$\Phi(x_1)\geq \log (\A\B)$, we have
$$\partial_+\ff(0)=\int_{x_1}^{x_2}\psi \Phi -\log(\A\B)\max_{[x_1,x_2]}\Psi\geq
\left(\Phi(x_1) -\log(\A\B)\right)
\max_{[x_1,x_2]}\Psi\geq 0.$$
This proves \eqref{*Pavel3}.
To prove \eqref{Pre-I}, we combine \eqref{Step4.5}, \eqref{P3}, and  \eqref{*Pavel3}.
This concludes the proof of Proposition \ref{Prop-LD-derrida-II-a}.
\end{proof}
\subsection{Large Deviations for the mean particle density}
 The mean particle density  is
 $$\frac{1}N\sum_{j=1}^N \tau_j=\tfrac1N H_N(N)=\tfrac 1N \wt H_N(1).$$
 The following proposition,  recalculated from   \cite[formula (3.12)]{derrida2003-Exact-LDP}, gives
explicit formula for the rate function of the mean particle density in the fan region of TASEP.
An equivalent result with a different proof appeared in \cite[Theorem 7]{Bryc-Wesolowski-2015-asep}.
\hide
{ \arxiv{
 The notation used in \cite{derrida2003-Exact-LDP} is
 $\rho_a=1/(1+\A)$, $\rho_b=\frac{\B}{1+\B}$ and
 $h(r,f,g)=h(r|f)+\log (f(1-f))-\log(g(1-g))$ (recalculated), and we use the same
 $\bar \rho$ in \eqref{rho-bar} that they introduced.
 According to \cite[formula (3.12)]{derrida2003-Exact-LDP}, the rate function on constant profile $f(x)=r x$ is
$$
   \mathcal{I}(r)=\begin{cases}
     h(r|\tfrac{1}{1+\A})+ \log \tfrac{\A}{(1+\A)^2}-\log (\bar\rho(1-\bar \rho)) & 0 \leq r\leq 1-\rho_a=\tfrac{\A}{1+\A}     \\
 h(r,1-r,\bar\rho)& \tfrac{\A}{1+\A} \leq r\leq \tfrac{1}{1+\B} \\
 h(r|\tfrac{\B}{1+\B})+ \log\tfrac{\B}{(1+\B)^2} -\log (\bar\rho(1-\bar \rho))   &  \tfrac{1}{1+\B} < r\leq 1   \end{cases}
$$
Since $K(\A,\B)=\log(\bar\rho(1-\bar\rho))$   we only need to verify that the middle expressions agree:
$$
h(r,1-r,\bar\rho)=r \log \tfrac{r}{1-r}+ (1-r) \log \tfrac{1-r}{r}+\log (r(1-r))-\log(\bar\rho (1-\bar\rho))
= 2 h(r|\tfrac12) -\log 4 - K(\A,\B).
$$
Discussion
in \cite[Appendix B]{derrida2003-Exact-LDP} identifies this expression as the rate function for the mean density.
 }
 }
 \begin{proposition}
   \label{Prop-B-W}
   If $\A\B\leq 1$, then the mean particle density $\frac1NH_N(N)$ satisfies the large deviation principle with the rate function
  \begin{equation}
  \label{I-BW}
  \mathsf{I}(r)=-K(\A,\B) +\begin{cases}
   h(r|\tfrac{1}{1+\A})  +\log \tfrac{\A}{(1+\A)^2}& 0\leq r< \tfrac{\A}{1+\A},\\
   2 h(r|\tfrac12)   +\log \tfrac14 &    \tfrac{\A}{1+\A}\leq r \leq \tfrac{1}{1+\B}, \\
   h(r|\tfrac{\B}{1+\B})  +\log \tfrac{\B}{(1+\B)^2} & \tfrac{1}{1+\B}< r \leq 1,
  \end{cases}
\end{equation}
with $K(\A,\B)$ given by \eqref{K:ab<1} and entropy $h(\,\cdot\,|\,\cdot\,)$ given by \eqref{H-entr}.
 \end{proposition}

Although the result is known, we re-derive formula \eqref{I-BW} from \eqref{I} for the special case of TASEP, as the  argument
  establishes   \eqref{K:ab<1}, and hence we will conclude the derivation of formula \eqref{Lab}  for $\A\B\leq  1$.

 \begin{proof}[Proof of Proposition \ref{Prop-B-W}]
 We use \eqref{I} and contraction principle. To avoid circular reasoning,
 we use \eqref{I} with $K_0(\A,\B)$ given by \eqref{I_0}. (In fact, we leave $K_0(\A,\B)$
  as a free parameter to be determined at the end of the proof.)

  Through the proof, we fix $r\in[0,1]$.
We write $\mathsf{I}(r)$ as the infimum over $m\in[0,1]$ and over all functions $f_1,f_2\in\mathcal{AC}_0$ such that $f_1(1)=r$, $f_2(1)=m$. The first step is to show that optimal $f_1,f_2$ are linear.
 To do so we note that since $\A\B\leq 1$, we have
 $$
 \log(\A\B)\min_{y\in[0,1]}\{f_1(y)-f_2(y)\}\geq \log(\A\B)\min\{0,(r-m)\}
 $$
 with equality on linear functions.
Therefore, the expression $\log(\A\B)\min_y\{f_1(y)-f_2(y)\}$ can only decrease if we replace
$f_1,f_2$ with a pair of linear functions $f(x)=rx$ and $g(x)=mx$. In view of convexity of $h$, this replacement also
decreases the integral in \eqref{I}. This shows that the optimal functions $f_1,f_2$ are indeed linear, $f(x)=rx$ and $g(x)=mx$.
We get %
\[
  \mathsf{I}(r)=-K_0(\A,\B) +h(r)+\min_{m\in[0,1]}\left\{h(m)+\max_{y\in[0,1]}\{  y \log(\A\B)(r-m)-\log \B (r-m) \}\right\}.
\]
 The maximum over $y\in[0,1]$ is attained at the end points of $[0,1]$ and since $\log(\A\B)\leq 0$, it is  either $(r-m) \log\A $ or $(m-r) \log \B$ depending on whether $m\geq r $ or $m<r$. (Recall that $r\in[0,1]$ is fixed.)

 In the first case, the infimum over $m\geq r$ is attained at
$$m= \begin{cases}
\tfrac{\A}{1+\A} & \mbox{if $r< \tfrac{\A}{1+\A}$},\\
r & \mbox{if $r\geq  \tfrac{\A}{1+\A}$},
\end{cases} $$
and gives
$$ \mathsf{I_1}(r)=-K_0(\A,\B)+ \begin{cases}
 h(r)+r \log \A-\log(1+\A)& r< \tfrac{\A}{1+\A},\\
   2 h(r) & r\geq  \tfrac{\A}{1+\A},
\end{cases}$$
with  $K_0(\A,\B)$ given by \eqref{I_0}. Note that
$$ h(r)+r \log \A-\log(1+\A) = h(r|\tfrac{1}{1+\A})+ \log \tfrac{\A}{(1+\A)^2}. $$

In the second  case, the minimum over $m\leq r$ is attained at
$$m= \begin{cases}
\tfrac{1}{1+\B} & \mbox{if $r> \tfrac{1}{1+\B}$},\\
r & \mbox{if $r\leq  \tfrac{1}{1+\B}$},
\end{cases}$$
and gives
$$\mathsf{I_2}(r)=-K_0(\A,\B)+ \begin{cases}
 h(r) -r \log (\B)+\log \tfrac{\B}{\B+1} & r> \tfrac{1}{1+\B}, \\
2 h(r) & r\leq  \tfrac{1}{1+\B}.
\end{cases} $$
Note that
$$ h(r) -r \log (\B)+\log \tfrac{\B}{\B+1}=h(r|\tfrac{\B}{1+\B})+\log \tfrac{\B}{(1+\B)^2}.$$
Also, note that $h(r)=h(r|\tfrac12)-\log 2$.
 Since we are interested in overall minimum over all $m\in[0,1]$, up to the additive normalizing constant $K_0(\A,\B)$,
 the resulting rate function is
  $$\mathsf{I}(r)=\min\{\mathsf{I_1}(r),\mathsf{I_2}(r)\}=-K_0(\A,\B)+\begin{cases}
   h(r|\tfrac{1}{1+\A})+ \log \tfrac{\A}{(1+\A)^2} &  0\leq r< \tfrac{\A}{1+\A}, \\
  h(r|\tfrac{\B}{1+\B})+\log \tfrac{\B}{(1+\B)^2} &  \tfrac{1}{1+\B}<r\leq 1,  \\
       2 h(r|2)-\log 4  &\mbox{ ortherwise},\\
  \end{cases}$$
   in agreement with \eqref{I-BW}.

The additive normalization constant $K_0(\A,\B)$ can now
be determined  from the condition that $\inf_{r\in[0,1]} \mathsf{I}(r)=0$. This gives \eqref{K:ab<1} as follows:

\[
K_0(\A,\B)=\inf_{r\in[0,1]}\begin{cases}
   h(r|\tfrac{1}{1+\A})+ \log \tfrac{\A}{(1+\A)^2} &  0\leq r< \tfrac{\A}{1+\A}, \\
  h(r|\tfrac{\B}{1+\B})+\log \tfrac{\B}{(1+\B)^2} &  \tfrac{1}{1+\B}<r\leq 1,  \\
       2 h(r|\tfrac12)-\log 4  & \tfrac{\A}{1+\A} \leq r \leq \tfrac{1}{1+\B}.
       \end{cases}
\]
Thus $K_0(\A,\B)=\log(\bar \rho(1-\bar\rho))$, matching \eqref{K:ab<1}. We note that this establishes   \eqref{Lab}  for $\A\B\leq 1$.
\arxiv{Here are the details of this calculation: $K_0(\A,\B)$ is the minimum of three expressions.
The first expression is
\begin{multline}\label{K-min1}
  \log \tfrac{\A}{(1+\A)^2}+\inf_{0\leq r<\A/(1+\A)} h(r|\tfrac{1}{1+\A})=
 \log \tfrac{\A}{(1+\A)^2}+ \begin{cases}
  h(\tfrac{1}{1+\A}|\tfrac{1}{1+\A}) & \A>1 \\
     h(\tfrac{\A}{1+\A}|\tfrac{1}{1+\A}) & \A\leq 1
  \end{cases}
 \\ =\begin{cases}
  \log \tfrac{\A}{(1+\A)^2} & \A>1 \\
  \tfrac{2 \A \log \A}{1+\A}-2 \log (1+\A) & \A\leq 1.
  \end{cases}
\end{multline}
The second expression is
\begin{multline}\label{K-min2}
  \log \tfrac{\B}{(1+\B)^2}+\inf_{\tfrac{1}{1+\B}< r\leq 1} h(r|\tfrac{\B}{1+\B})=
 \log \tfrac{\B}{(1+\B)^2}+ \begin{cases}
  h(\tfrac{\B}{1+\B}|\tfrac{\B}{1+\B}) & \B>1 \\
     h(\tfrac{1}{1+\B}|\tfrac{\B}{1+\B}) & \B\leq 1
  \end{cases}
 \\ =\begin{cases}
  \log \tfrac{\B}{(1+\B)^2} & \B>1 \\
 \tfrac{2 \B \log \B}{1+\B}-2 \log (1+\B) & \B\leq 1,
  \end{cases}
\end{multline}
and the third expression is
\begin{equation}
  \label{K-min3}
  \inf_{\tfrac{\A}{1+\A}\leq r\leq \tfrac{1}{1+b}} 2 h(r|\tfrac12)-\log 4=\begin{cases}
    \infty & \A>1 \mbox{ or } \B>1, \\
   -\log 4 & \A\leq 1, \B\leq 1.
  \end{cases}
  \end{equation}
  Noting that
$$
\frac{d}{d\B}\tfrac{2 \B \log \B}{1+\B}-2 \log (1+\B)=\frac{2 \log \B}{(\B+1)^2}<0 \mbox{ for $\B\in(0,1)$},
$$
we see that if $\B<1$ then the expression \eqref{K-min2} is larger than its value $-2 \log 2$ at $\B=1$.
On the other hand, for $\B>1$, the expression \eqref{K-min2} decreases, so it is less than its value $-2 \log 2$ at $\B=1$, and it is also less than the value of the expression \eqref{K-min1} when $\B>\A\geq 1$.

Since the same reasoning applies to \eqref{K-min1}, we see that for $(\A,\B)$ in each of the three regions of the phase diagram in Fig. \ref{Fig2} we get $K_0(\A,\B)=\log (\bar\rho(\A,\B)(1-\bar\rho(\A,\B)))$, as listed in \eqref{rho-bar}.
}
\end{proof}

 The LDP for the mean particle density in the shock region follows from Proposition \ref{Prop-LD-derrida-I} by the contraction principle.
The following proposition, recalculated from \cite[(B.8)]{derrida2003-Exact-LDP}, gives an
explicit formula for the rate function. 
  \begin{proposition}
   \label{Prop-B-W++}
   If $\A\B>1$, then
   the mean particle density $\frac1NH_N(N)$ satisfies the large deviation principle with the rate function
\begin{equation}\label{I-BW++}
    \mathsf{I}(r)=\begin{cases}
    h(r|\tfrac{1}{1+\A})+\log\tfrac{\A}{(1+\A)^2}-K(\A,\B) & 0\leq r \leq \tfrac{1}{1+\B}, \\
     r \log \tfrac{\A}{\B}+ \log \tfrac{\B}{(1+\A)(1+\B)}-K(\A,\B)& \tfrac{1}{1+\B}\leq  r \leq \tfrac{\A}{1+\A}, \\
    h(r|\tfrac{\B}{1+\B})+\log\tfrac{\B}{(1+\B)^2}-K(\A,\B) &  \tfrac{\A}{1+\A} \leq r\leq 1,
  \end{cases}
\end{equation}
with $\mathsf{I}=\infty$ for $r\not\in[0,1]$.
 \end{proposition}
\begin{proof} The proof of this formula appears in \cite[Section 3.6 and Appendix B, Case 2]{derrida2003-Exact-LDP} and is omitted.
\end{proof}

\arxiv{
Here are some more details for completeness. This is an expanded version of the calculations summarized on page 809 of \cite{derrida2003-Exact-LDP}.
   We use \eqref{I-Derrida} and write $\mathsf{I}(r)$ as the infimum over $y\in[0,1]$ and over $f\in\mathcal{AC}_0$ with $f(1)=r$.
Due to the convexity of $h(\cdot)$, $f$ must be piecewise linear on $[0,y]$ and on $[y,1]$, so we have
$$
f(x)=\begin{cases}
  r_a x & x\leq y \\
  r_a y+r_b (x-y) & y\leq x\leq 1
\end{cases}
$$
where $0\leq r_a,r_b\leq 1$ and   $r_a y+r_b(1-y)=r$. This gives
\begin{multline*}
  \mathsf{I}(r) =-K(\A,\B)\\+\min\Big\{  y\left(r_a \log \tfrac{\A  r_a}{1+\A}+(1-r_a) \log \tfrac{1-r_a}{1+\A}\right)+
  (1-y)\left(r_b \log \tfrac{ r_b}{1+\B}+(1-r_b) \log \tfrac{(1-r_b)\B}{1+\B}\right):\\
  0\leq y, r_a, r_b \leq 1, \; y r_a+(1-y)r_b=r\Big\}
\end{multline*}

We first seek a minimizer with $y\in(0,1)$. Using Lagrange multipliers with
$$F(r_a,r_b,y,\la)=  y\left(r_a \log \tfrac{\A  r_a}{1+\A}+(1-r_a) \log \tfrac{1-r_a}{1+\A}\right)+
  (1-y)\left(r_b \log \tfrac{ r_b}{1+\B}+(1-r_b) \log \tfrac{(1-r_b)\B}{1+\B}\right)-\la(y r_a+(1-y)r_b-r)$$
\begin{itemize}
  \item from $\frac{\partial F}{\partial r_a}=0$ we get $r_a=\frac{e^\la}{e^\la+\A}$;
  \item from  $\frac{\partial F}{\partial r_b}=0$ we get $r_b=\frac{\B e^\la}{1+\B e^\la}$;
  \item from  $\frac{\partial F}{\partial y}=0$ we get $\la=\log \A/\B$
\end{itemize}
Thus $r_a=\frac{1}{1+\B}$, $r_b=\frac{\A}{1+\A}$ and $y\in(0,1)$ is possible only when $\frac{1}{1+\B}<r<\frac{\A}{1+\A}$. In this case, $F(r_a,r_b,y,0)= r \log \tfrac{\A}{\B}+ \log \tfrac{\B}{(1+\A)(1+\B)}$, giving the formula we seek for $\mathsf I$  in this interval.
If $r\le 1/(1+\B)$, the minimum is attained at $y=1$ with $r_a=r$, giving
$F(r,r_b,y,0)=h(r|\frac{1}{1+\A})+\log \frac{\A}{(1+\A)^2}$.
If $r\ge \A/(1+\A)$, the minimum is attained at $y=0$ with $r_b=r$, giving
$F(r_a,r,y,0)=+h\left(r\middle |\frac{\B}{1+\B}\right)+\log \frac{\B}{(1+\B)^2}$.

Following  \cite{derrida2003-Exact-LDP} it might be worth noting that the two convex curves
$r\mapsto  F(r,r_b,y,0)$ and $r\mapsto  F(r_a,r,y,0)$ intersect at
$$r^* = \frac{ \log\tfrac{(1 + \A) \B}{1 + \B}}{\log(\A \B)}  $$
and that $\frac{1}{1+\B}<r^*<\frac{\A}{1+\A}$.
For example, writing $\A=\frac{t}{\B}$ with $x>0$, inequality $\frac{1}{1+\B}<r^*$ is equivalent to
$$
\log (t+\B+x)-\log(1+\B)- \frac{\log t}{1+\B}>0
$$
This expression is $0$ at $t=1$ and its derivative with respect to $t$, given by
$\frac{1}{\B+t }-\frac{1}{(1+\B)t}$, is positive for $t>1$.

}
We remark that on the coexistence line $\A=\B>1$, the rate function is zero on  the entire interval
$[1/(1+\A),\A/(1+\A)]$. This is consistent with
\cite[Theorem 1.6]{wang2023askey} which implies that mean particle density $\frac1N H_N(N)$ converges in
 distribution to the uniform law on this interval. (Shocks on the coexistence line for open ASEP
 were also described in
 \cite{derrida1997shock,derrida93exact,schutz1993phase}.)



\appendix
\section{Integrability Lemma}
 The L\'evy-Ottaviani maximal inequality and tail integration give the following bound:

\begin{lemma} For $\C>0$, we have
  \label{Lem:Levy}
  \begin{equation}
    \label{Levy-Ott}
    \EE_{\mathrm{rw}}\left[\C^{-\min_{0\leq j \leq N}(S_1(j)-S_2(j))}\right]\leq 1+2
   \, \EE_{\mathrm{rw}}\left[\C^{|S_1(N)-S_2(N)|}\right],
  \end{equation}
  where $\EE_{\mathrm{rw}}$ is expectation with respect to the law of the two independent Bernoulli$(1/2)$ random walks.
\end{lemma}

\begin{proof}
  Recall that for sums of independent symmetric random variables $\{S_1(j)-S_2(j)\}$, the L\'evy-Ottaviani maximal inequality says that
  $$
  \PP_{\mathrm{rw}}\left(\max_{1\leq j \leq N}\left|S_1(j)-S_2(j)\right|>t \right)\leq 2  \PP_{\mathrm{rw}}\left( \left|S_1(N)-S_2(N)\right|>t \right), \quad t\geq 0.
  $$
  Since $$0\leq -\min_{0\leq j \leq N}(S_1(j)-S_2(j))\leq \max_{1\leq j \leq N}|S_1(j)-S_2(j)|,$$ we see that if
  $\C\leq 1$ then the left hand side of \eqref{Levy-Ott} is bounded by $1$. And if $\C>1$,
then by the above bound and tail integration, we have
\begin{multline*}
  \EE_{\mathrm{rw}}\left[\C^{-\min_{0\leq j \leq N}(S_1(j)-S_2(j))}\right]\leq
   \EE_{\mathrm{rw}}\left[\C^{\max_{1\leq j \leq N}\left|S_1(j)-S_2(j)\right|}\right]
  \\ = 1+\log(\C) \int_0^\infty \C^t
  \PP_{\mathrm{rw}}\left(\max_{1\leq j \leq N}\left|S_1(j)-S_2(j)\right|>t \right) dt
   \\
   \leq
   1+2 \log(\C) \int_0^\infty \C^t
  \PP_{\mathrm{rw}}\left( \left|S_1(N)-S_2(N)\right|>t \right) dt
  \\<
  2\left(1+ \log(\C) \int_0^\infty \C^t
  \PP_{\mathrm{rw}}\left( \left|S_1(N)-S_2(N)\right|>t \right) dt \right)
  \\ =2
   \, \EE_{\mathrm{rw}}\left[\C^{|S_1(N)-S_2(N)|}\right].
\end{multline*}
\end{proof}
\begin{remark}\label{Rem:howtouse}
Note that   elementary inequalities $e^{|x|}\leq e^x+e^{-x}$ and  $1+\cosh(x)\leq 2 e^{x^2/2}$ together with independence give
\begin{multline*}
     \EE_{\mathrm{rw}}\left[e^{\la |S_1(N)-S_2(N)|}\right]\leq
 \EE_{\mathrm{rw}}\left[e^{\la (S_1(N)-S_2(N))}\right]+\EE_{\mathrm{rw}}\left[e^{\la (S_2(N)-S_1(N))}\right]\\=2 \left(\frac{1+\cosh(\la)}{2}\right)^N\leq
2 e^{\la^2 N/2}.
\end{multline*}
In particular,  using this with $\la=2/\sqrt{N}$ we get a bound
$$\sup_N \EE_{\mathrm{rw}}\left[\mathcal{E}^2(\vv W_-\topp N)\right]<\infty$$
that is used in the proof of Theorem \ref{Thm2.1}, and  using this bound with $\la=-8\log(\A\B)$ and $\la=8\log \B$, we get a bound
\begin{multline*}
 \sup_N \frac1N \log   \EE_{\mathrm{rw}}\left[ e^{ 4 N \log \tilde g(X_N)}\right]=
\sup_N \frac1N \log   \EE_{\mathrm{rw}}\left[ e^{4\log  g(\vv S_1,\vv S_2)}\right]
\\ \leq \sup_N   \frac1{2N} \log \EE_{\mathrm{rw}}\left[(\A\B)^{-8 \min_{0\leq j\leq N}\{S_1(j)-S_2(j)\}  }\right]+\sup_N\frac1{2N}
 \log \EE_{\mathrm{rw}}\left[\B^{8(S_1(N)-S_2(N)) }\right]
 \\\leq
  \sup_N   \frac1{2N} \log \left(1+2\EE_{\mathrm{rw}}\left[\C^{|S_1(N)-S_2(N)|  }\right]\right)+\sup_N\frac1{2N}
 \log \EE_{\mathrm{rw}}\left[\B^{8(S_1(N)-S_2(N)) }\right] <\infty,
\end{multline*}
which is used in the proof of Theorem \ref{Thm:LDP}.
\end{remark}
\arxiv
{
\section{Further discussion of the large deviations when  $\A\B<1$}\label{DLS2US}
According to \cite[(3.3) and (3.6)]{derrida2003-Exact-LDP}, the rate function  in the fan region $\A\B\leq 1$ is
\begin{equation}
  \label{I-Derrida2}
  \mathcal{I}(f)=\sup_{G}\left\{\int_0^1\left( f'(x)\log( f'(x)G(x))+(1-f'(x))\log(1-f'(x))(1-G(x))\right) dx \right\}-K(\A,\B),
\end{equation}
where the supremum is over all nondecreasing functions $G$ such that $G(0)=\frac{\A}{1+\A}$ and
$G(1)=\frac{1}{1+\B}$ and
 $K(\A,\B)$ is given by \eqref{K:ab<1}.
According to \cite{derrida2003-Exact-LDP} the supremum in  \eqref{I-Derrida2} is attained at
$G_*$ given by \eqref{*G*}.
Of course, \eqref{I-Derrida2} is  $$\mathcal{I}(f)=\int_0^1 h(f'(x))dx -K(\A,\B)+  \sup_G J_*(f, G),$$
where $J_*(f,G)$ was defined in \eqref{JG*}. This differs slightly from \eqref{I-Derrida2*} where $J_*(\wt f,G_*)$ is used instead of  $J_*( f,G_*)$.

  Any monotone function can be represented as a pointwise limit of a sequence of continuous functions,
and the functional $J_*(f,\cdot)$ is continuous with respect to this convergence. In addition, we can write inequalities instead of equalities at the endpoints.
 To reconcile the formulas \eqref{I-Derrida2} and \eqref{I-Derrida2*} for the rate function, we fix $f\in\mathcal{AC}_0$ with continuous $f'\in[0,1]$ and such that $G_*'$ is continuous. We will  show that
  $\sup_{G} J_*(f,G)=J_*(f,G_*)=J_*(\wt f,G_*)$  by verifying that for any   smooth nondecreasing $G$ that satisfies
  \begin{equation}\label{G-class}
       \frac{\A}{1+\A} \leq G(0)\leq G(1)\leq \frac{1}{1+\B}
  \end{equation}
 we have
  \begin{equation}\label{Pavel1}
 J_*(f,G)\leq J_*(\wt f,G)\leq J_*(\wt f,G_*)=J_*( f,G_*),
  \end{equation}
hence the supremum $\sup_{G} J_*(f,G)$ is attained at $G_*$ and is equal to $J_*(\wt f,G_*)$.

The argument is similar to \cite{derrida2003-Exact-LDP}.
To verify the first inequality in~\eqref{Pavel1}, we use the fact that $f - \wt f\geq 0$ and that $f - \wt f$ vanishes at the endpoints
$x=0,1$. Since $G'\geq 0$, integrating by parts, we obtain \begin{align}
 J_*(f,G) =& \Big(f\log G +(1-f)\log(1-G)\Big)\Big|_{0}^1 - \int_0^1 f(x) \frac{G'(x)}{G(x)\big(1-G(x)\big)}dx\notag\\
 \leq&
 \Big(\wt f\log G +(1-\wt f)\log(1-G)\Big)\Big|_{0}^1 - \int_0^1 \wt f(x) \frac{G'(x)}{G(x)\big(1-G(x)\big)}dx=J_*(\wt f,G).\label{Pavel2}
  \end{align}

To prove the second inequality in~\eqref{Pavel1} we note that for each $s \in [0,1]$ the function $t \mapsto s\log t +(1-s) \log(1-t)$ is concave on $[0,1]$ and attains its maximum at $t=s$. Moreover, its maximum on the interval $R=[\frac{a}{a+1},\frac{1}{b+1}]$ is attained at the point of $R$ that is the closest to $s$. Therefore, for each $x \in [0,1]$ the maximal value of $t \mapsto \wt f'(x)\log t +(1-\wt f'(x)) \log(1-t)$ on $R$ is attained at $G_*(x)$. Thus, for any $G$ satisfying~\eqref{G-class} we have
$$
\wt f'(x)\log G(x) +(1-\wt f'(x)) \log(1-G(x)) \leq
\wt f'(x)\log G_*(x) +(1-\wt f'(x)) \log(1-G_*(x)).
$$
The second inequality in~\eqref{Pavel1}  follows via integration.

To prove the last part of~\eqref{Pavel1},  note that $G_*$ is constant in a neighborhood of any point where $f> \wt f$, therefore $(f-\wt f)G_*'=0$ on the entire interval $[0,1]$. Thus,~\eqref{Pavel2} becomes an equality for $G=G_*$.

We now extend this to any $f$ in $\mathcal{AC}_0$ with $f'\in[0,1]$. We can find a sequence of smooth functions $f_n$ from $\mathcal{AC}_0$ such that $f'_n\in [0,1]$ and $f'_n$ converge to $f'$ a.\,e. on $[0,1]$. Then $f_n$ converge to $f$ uniformly, and thus $\wt f_n$ converge to $\wt f$ uniformly on $[0,1]$. The functions $\wt f_n$ and $\wt f$ are convex, therefore the uniform convergence implies convergence of the derivatives $\wt f’_n$ to $\wt f’$ at the points of continuity of $\tilde f’$ (in particular, a.e.). Then $J_*(f_n,G)$ converge to $J_*(f,G)$ and $J_*(\wt f_n,G)$ converge to $J_*(\wt f,G)$ uniformly with respect to $G$. Therefore,
$$
\sup_G J_*(f,G) = \lim_{n \to \infty}\sup_G J_*(f_n,G)  =  \lim_{n \to \infty}\sup_G J_*(\wt f_n,G) =
\sup_G J_*(\wt f,G),
$$
and the last supremum is attained at $G_*$.
}

\subsection*{Acknowledgement}
 This research benefited from discussions with Guillaume Barraquand, Ivan Corwin,
and Yizao Wang. This work was supported by a Simons grant   (703475, WB).

\begin{thebibliography}{10}

\bibitem{barraquand2024stationary}
Guillaume Barraquand, Ivan Corwin, and Zongrui Yang, \emph{Stationary measures
  for integrable polymers on a strip}, Invent. Math. \textbf{237} (2024),
  no.~3, 1567--1641, \url{https://arxiv.org/pdf/2306.05983}. \MR{4777093}  

\bibitem{barraquand2022steady}
Guillaume Barraquand and Pierre {Le Doussal}, \emph{Steady state of the {KPZ}
  equation on an interval and {L}iouville quantum mechanics}, Europhysics
  Letters \textbf{137} (2022), no.~6, 61003, ArXiv preprint with Supplementary
  material: \url{https://arxiv.org/abs/2105.15178}.

\bibitem{Barraquand2023Motzkin}
Guillaume Barraquand and Pierre Le~Doussal, \emph{Stationary measures of the
  {KPZ} equation on an interval from {E}naud-{D}errida's matrix product ansatz
  representation}, J. Phys. A \textbf{56} (2023), no.~14, Paper No. 144003, 14.
  \MR{4562516}

\bibitem{Bertini-2007}
L.~Bertini, A.~De~Sole, D.~Gabrielli, G.~Jona-Lasinio, and C.~Landim,
  \emph{Stochastic interacting particle systems out of equilibrium}, J. Stat.
  Mech. Theory Exp. (2007), no.~7, P07014, 35. \MR{2335695}

\bibitem{billingsley99convergence}
Patrick Billingsley, \emph{Convergence of probability measures}, second ed.,
  Wiley Series in Probability and Statistics: Probability and Statistics, John
  Wiley \& Sons Inc., New York, 1999, A Wiley-Interscience Publication.
  \MR{1700749 (2000e:60008)}

\bibitem{brak2006combinatorial}
Richard Brak, Sylvie Corteel, John Essam, Robert Parviainen, and Andrew
  Rechnitzer, \emph{A combinatorial derivation of the {PASEP} stationary
  state}, Electron. J. Combin. \textbf{13} (2006), no.~1, Research Paper 108,
  23. \MR{2274323}

\bibitem{Bryc-Wang-Wesolowski-2022}
W{\l}odek Bryc, Yizao Wang, and Jacek Weso{\l}owski, \emph{From the asymmetric
  simple exclusion processes to the stationary measures of the {KPZ} fixed
  point on an interval}, Ann. Inst. Henri Poincar\'e{} Probab. Stat.
  \textbf{59} (2023), no.~4, 2257--2284,
  \url{https://arxiv.org/abs/2202.11869}. \MR{4663522}

\bibitem{Bryc-Wesolowski-2015-asep}
W{\l}odek Bryc and Jacek Weso{\l}owski, \emph{Asymmetric simple exclusion
  process with open boundaries and quadratic harnesses}, J. Stat. Phys.
  \textbf{167} (2017), no.~2, 383--415, \url{http://arxiv.org/abs/1511.01163}.
  \MR{3626634}

\bibitem{Bryc-Wang-2023b}
W{\l}odzimierz Bryc and Yizao Wang, \emph{Fluctuations of random {M}otzkin
  paths {II}}, ALEA Lat. Am. J. Probab. Math. Stat. \textbf{21} (2024), no.~1,
  73--94, \url{http://arxiv.org/abs/2304.12975}. \MR{4703770}

\bibitem{DZ-1998-large}
Amir Dembo and Ofer Zeitouni, \emph{Large deviations techniques and
  applications}, second ed., Applications of Mathematics (New York), vol.~38,
  Springer-Verlag, New York, 1998. \MR{1619036}

\bibitem{Derrida-Domany-Mujamel-1992}
B.~Derrida, E.~Domany, and D.~Mukamel, \emph{An exact solution of a
  one-dimensional asymmetric exclusion model with open boundaries}, J. Statist.
  Phys. \textbf{69} (1992), no.~3-4, 667--687. \MR{1193854}

\bibitem{derrida04asymmetric}
B.~Derrida, C.~Enaud, and J.~L. Lebowitz, \emph{The asymmetric exclusion
  process and {B}rownian excursions}, J. Statist. Phys. \textbf{115} (2004),
  no.~1-2, 365--382. \MR{2070099}

\bibitem{Derrida-Evans-1993}
B.~Derrida and M.~R. Evans, \emph{Exact correlation functions in an asymmetric
  exclusion model with open boundaries}, J. Physique I \textbf{3} (1993),
  no.~2, 311--322. \MR{1215845}

\bibitem{Derrida-DEHP-1993}
B.~Derrida, M.~R. Evans, V.~Hakim, and V.~Pasquier, \emph{Exact solution of a
  {$1$}d asymmetric exclusion model using a matrix formulation}, J. Phys. A
  \textbf{26} (1993), no.~7, 1493--1517. \MR{1219679}

\bibitem{derrida1997shock}
B.~Derrida, J.~L. Lebowitz, and E.~R. Speer, \emph{Shock profiles for the
  asymmetric simple exclusion process in one dimension}, J. Statist. Phys.
  \textbf{89} (1997), no.~1-2, 135--167, Dedicated to Bernard Jancovici.
  \MR{1492490}

\bibitem{derrida2003-Exact-LDP}
\bysame, \emph{Exact large deviation functional of a stationary open driven
  diffusive system: the asymmetric exclusion process}, J. Statist. Phys.
  \textbf{110} (2003), no.~3-6, 775--810, Special issue in honor of Michael E.
  Fisher's 70th birthday (Piscataway, NJ, 2001). \MR{1964689}

\bibitem{derrida2002exact}
B~Derrida, JL~Lebowitz, and ER~Speer, \emph{Exact free energy functional for a
  driven diffusive open stationary nonequilibrium system}, Physical Review
  Letters \textbf{89} (2002), no.~3, 030601.

\bibitem{derrida06matrix}
Bernard Derrida, \emph{Matrix ansatz and large deviations of the density in
  exclusion processes}, International {C}ongress of {M}athematicians. {V}ol.
  {III}, Eur. Math. Soc., Z\"urich, 2006, pp.~367--382. \MR{2275686}

\bibitem{derrida07nonequilibrium}
\bysame, \emph{Non-equilibrium steady states: fluctuations and large deviations
  of the density and of the current}, J. Stat. Mech. Theory Exp. \textbf{2007}
  (2007), no.~7, P07023, 45. \MR{2335699}

\bibitem{derrida93exact}
Bernard Derrida, Martin~R. Evans, Vincent Hakim, and Vincent Pasquier,
  \emph{Exact solution of a {$1$}{D} asymmetric exclusion model using a matrix
  formulation}, J. Phys. A \textbf{26} (1993), no.~7, 1493--1517. \MR{1219679}

\bibitem{duchi2005combinatorial}
Enrica Duchi and Gilles Schaeffer, \emph{A combinatorial approach to jumping
  particles}, J. Combin. Theory Ser. A \textbf{110} (2005), no.~1, 1--29.
  \MR{2128962}

\bibitem{Dupuis-Ellis1997weak}
Paul Dupuis and Richard~S. Ellis, \emph{A weak convergence approach to the
  theory of large deviations}, Wiley Series in Probability and Statistics:
  Probability and Statistics, John Wiley \& Sons, Inc., New York, 1997, A
  Wiley-Interscience Publication. \MR{1431744}

\bibitem{Liggett-1975}
Thomas~M. Liggett, \emph{Ergodic theorems for the asymmetric simple exclusion
  process}, Trans. Amer. Math. Soc. \textbf{213} (1975), 237--261. \MR{410986}

\bibitem{nestoridi2023approximating}
Evita Nestoridi and Dominik Schmid, \emph{Approximating the stationary
  distribution of the {ASEP} with open boundaries}, Comm. Math. Phys.
  \textbf{405} (2024), no.~8, Paper No. 176, 64. \MR{4777077}

\bibitem{schutz1993phase}
Gunter Sch{\"u}tz and Eytan Domany, \emph{Phase transitions in an exactly
  soluble one-dimensional exclusion process}, Journal of Statistical Physics
  \textbf{72} (1993), no.~1-2, 277--296.

\bibitem{wang2023askey}
Yizao Wang, Jacek Weso{\l}owski, and Zongrui Yang, \emph{Askey-{W}ilson signed
  measures and open {ASEP} in the shock region}, Int. Math. Res. Not. IMRN
  (2024), no.~15, 11104--11134. \MR{4782793}

\bibitem{wang2024asymmetric}
Yizao Wang and Zongrui Yang, \emph{From asymmetric simple exclusion processes
  with open boundaries to stationary measures of open {KPZ} fixed point: the
  shock region}, 2024, \url{https://arxiv.org/abs/2406.09252}.

\end{thebibliography}
\providecommand{\bysame}{\leavevmode\hbox to3em{\hrulefill}\thinspace}
\renewcommand{\MR}[1]{\href{http://www.ams.org/mathscinet-getitem?mr=#1}{MR#1}}

\end{document}